\documentclass[12pt]{amsart}

\makeatletter \@addtoreset{equation}{section}

\makeatother \makeatletter
\newtheorem{theorem}{Theorem}[section]
\newtheorem{corollary}[theorem]{Corollary}
\newtheorem{definition}[theorem]{Definition}
\newtheorem{lemma}[theorem]{Lemma}
\newtheorem{proposition}[theorem]{Proposition}
\newtheorem{remark}[theorem]{Remark}

\newtheorem*{theorem*}{Theorem}

\def\pr1{\prod\hskip -2.07ex * \hskip 0.9 ex}

\usepackage{latexsym} \usepackage{amsmath} \usepackage{amssymb}
\usepackage{amsthm}

\usepackage{cancel}\usepackage{color}
\begin{document}
\noindent
\title{A Strong Version of the Hilbert Nullstellensatz for slice regular  polynomials in several quaternionic variables}
\author{Anna Gori $^1$} 
 \thanks{$^1$ Dipartimento di Matematica - Universit\`a di Milano,                
 Via Saldini 50, 20133  Milano, Italy}
 \author{Giulia Sarfatti $^2$}
 \thanks{ $^2$ DIISM - Universit\`a Politecnica delle Marche,               
 	Via Brecce Bianche 12,  60131, Ancona, Italy}
 \author{Fabio Vlacci $^3$}
 \thanks{$^3$ MIGe  - Universit\`a di Trieste, Piazzale Europa 1, 34100, Trieste, Italy}

 \begin{abstract}
In this paper we prove a strong version of the  Hilbert Nullstellensatz in the ring
$\mathbb H[q_1,\ldots,q_n]$ of slice regular polynomials in several
quaternionic variables. Our proof deeply depends on a detailed
analysis of the common zeros of slice regular polynomials which belong to an ideal in
$\mathbb H[q_1,\ldots,q_n]$. This study motivates the introduction of
a new notion of algebraic set in the quaternionic setting, which
allows us to define a Zariski-type topology on $\mathbb H^n$.

%in two (or several) quaternionic variables.

 \end{abstract}
\maketitle
\noindent
%\section*{Acknowledgments}
% The authors are partially supported by:
%%The first and third authors were partly supported
%GNSAGA-INdAM via the project ``Hypercomplex function theory and applications''; by MUR, via the project Finanziamento Premiale FOE 2014 ``Splines for accUrate NumeRics: adaptIve models for Simulation Environments''; the first author is also partially supported by MUR project PRIN 2022 ``Real and Complex Manifolds: Geometry and Holomorphic Dynamics'', the second and the third authors are also partially supported by MUR project PRIN 2022 ``Interactions between Geometric Structures
%and Function Theories''.

\section*{Acknowledgments}
The authors are partially supported by:
%The first and third authors were partly supported
GNSAGA-INdAM via the project ``Hypercomplex function theory and
applications''; the first author is also partially supported by MUR
project PRIN 2022 ``Real and Complex Manifolds: Geometry and
Holomorphic Dynamics'', the second and the third authors are also
partially supported by MUR projects PRIN 2022 ``Interactions between
Geometric Structures and Function Theories'' and Finanziamento
Premiale FOE 2014 ``Splines for accUrate NumeRics: adaptIve models for
Simulation Environments''.

\section{Introduction}

In the usual framework of $\mathbb{C}^n$, the Hilbert Nullstellensatz, 
both in its weak and  (equivalent) strong version,
represents a  highly relevant result in 
Algebraic Geometry; indeed  it establishes a correspondence between subsets of $\mathbb{C}^n$ and ideals in the ring of polynomials in $\mathbb{C}[z_1,\ldots,z_n]$ that paves the way to the 
%foundation and the i
introduction of
foundational concepts such as the notion of % related to %%the notion of 
algebraic varieties.\\
The major motivation of our research % and investigations 
is %can be found in the  aim  of
the possibility of  introducing a concept of algebraic variety also in the non-commutative
framework, and in particular in the quaternionic setting.
%which might provide this correspondence
There have been several attempts to formulate  weak and
strong versions of the analog of the Hilbert Nullstellensatz in a non-commutative framework;
among other results in this area, we recall
\cite{BR},
 the contributions % of Alon and Paran 
in \cite{israeliani-1} for  quaternionic polynomials with central
variables, and the ones  in \cite{Nul1} for slice regular
  polynomials in two quaternionic variables.
The main goal of the present paper is to establish a
Nullstellesatz-type theorem for slice regular polynomials in several
quaternionic variables which extends the results proved in \cite{Nul1}.\\
%in order to  give basic 
%contributions to the development of the foundations of  
 %algebraic geometry in this setting.
Slice regular  polynomials in $n$ quaternionic variables are polynomial functions
$P:{\mathbb{H}^n}  \to \mathbb{H}$, of the form
	\[	 (q_1,\ldots, q_n)\mapsto P(q_1,\ldots, q_n)=
	\sum_{ \substack{\ell_1=0,\ldots, L_1  \\
	\cdots\\ \ell_n=0,\ldots, L_n} }{q_1}^{\ell_1}\cdots {q_n}^{\ell_n}a_{\ell_1,\ldots,\ell_n} \]
	with $a_{\ell_1,\ldots,\ell_n}\in\mathbb{H}$, 
where $\deg_{{q_{\ell_j}}}P:=L_j$.\\	
These polynomials belong to the class of {\em slice regular} functions in several quaternionic variables, according to Definition 2.5 in \cite{severalvariables}.
The set $\mathbb{H}[q_1,\ldots,q_n]$ of slice regular polynomials in $n$ quaternionic variables 
can be endowed with an
appropriate notion of (non--commutative) product, the so called {\em slice product}, denoted by the symbol $*$; in this way, $(\mathbb{H}[q_1,\ldots,q_n], +,*)$ has the structure of a non-commutative ring.
%used in the one-variable
%case. %We refer to \cite{severalvariables} for a precise definition,
%We shortly  recall here how it works for
%% monomials, and use the fact that we can extend it by linearity to define it for polynomials and 
%slice regular polynomials in two quaternionic variables.
%If $P(q_1,q_2)=\sum_{ \substack{n=0,\ldots,N_1  \\ m=0,\ldots,N_2} } q_1^nq_2^m a_{n,m}$ and
%        $Q(q)=\sum_{ \substack{n=0,\ldots,L_1  \\ m=0,\ldots,L_2} }q_1^nq_2^m b_{n,m}$ are two
%        slice regular polynomials,
%        then the {\em $*$-product} of $P$ and $Q$ is the slice regular polynomial
%        defined by
%$$
%P*Q (q_1,q_2):=\sum_{ \substack{n=0,\ldots,N_1+L_1  \\ m=0,\ldots,N_2+L_2}}q_1^nq_2^m\sum_{ \substack{r=0,\ldots,n  \\ s=0,\ldots,m} }a_{r,s} b_{n-r,m-s}.
%$$

The set of slice regular polynomials vanishing on a given subset $Z$ of $\mathbb{H}^n$ is not in general an ideal in 
$\mathbb{H}[q_1,\ldots,q_n]$.
In order to  establish a direct link between  subsets  of $\mathbb{H}^n$
and  (right) ideals of $\mathbb{H}[q_1,\ldots,q_n]$,
we  define  $\mathcal{J}(Z)$ to be the  right  ideal 
 % associated with a given subset $Z$ of $\mathbb{H}^n$, namely the  ideal
 generated by
slice regular polynomials in several variables which vanish on the subset $Z$. 
On the other hand  starting from  a right ideal $I$ in $\mathbb{H}[q_1,\ldots,q_n]$, we 
define $\mathcal{V}(I)$ to be the set of common zeros of slice regular polynomials
in $I$. \\
%The subsets of $\mathbb{H}^n$  considered 
In  \cite{israeliani-1}, the authors consider  the vanishing sets  of quaternionic polynomials with central variables; this means that these sets in $\mathbb{H}^n$
consist only of  points
with commuting components and thus have rather limited geometric properties.
In our setting, 
% we can  actually consider any 
vanishing sets of slice regular quaternionic  polynomials can also contain points with non--commutative components.
%, therefore also sets of points
%with not-commuting components.
A reasonable approach to study zeros of a polynomial is to begin with the description of zeros of its factors.
 However,
%when studying 
the
vanishing set of a  $*$-product of slice regular polynomials,
%the problem that zeros of this type
% it turns out that  this set 
is not always immediately deducible from
the vanishing sets of the polynomial factors. % involved in the $*$-product.
After some preliminary results, the present paper focuses on two
classes of  vanishing sets  %in several quaternionic variables
%; in particular, different types of vanishing sets are considered, primarily those
 which have specific
symmetries and are preserved
 under the $*$-product. \\
\noindent For any $(a_1,\ldots,a_n)\in \mathbb {H}^n$, we define
\[\mathbb S_{(a_1,\ldots,a_n)}=\{(g^{-1}a_1g, \ldots, g^{-1}a_ng) : g \in \mathbb{H}\setminus\{0\} \}\]
%by $\mathbb S_{(a_1,\ldots,a_n)}
to be the set obtained by simultaneously  rotating each coordinate of the point $(a_1,\ldots,a_n)$.
 Any set of this form will be called a {\em spherical set} in analogy with the terminology adopted 
 in the one quaternionic variable setting. 
 In particular, if $(a_1,\ldots,a_n)$ has commuting components, i.e.  $a_la_m=a_ma_l$ for any  $l,m$, then the spherical set will be said an {\em arranged} spherical set. 
Any set of the form $\mathbb S_{(a_1,\cdots,a_k)}\times \{a_{k+1}\} \times \cdots\times \{  a_{n}\}$,  with $\mathbb S_{(a_1,\cdots,a_k)}$ an arranged spherical set and  
$a_la_m=a_ma_l$ for any  $l,m\geq k+1$,
will be called a {\em balloon}.

\noindent In order to prove that balloons and subsets of the form $\mathbb S_{a_1}\times \cdots \times \mathbb S_{a_k}\times \cdots \times \{a_{k+1}\}\times \{a_n\}$ are zero loci of right ideals in $\mathbb H[q_1,\ldots,q_n]$, we provide an explicit description of slice regular polynomials vanishing on these sets.

\noindent Given a right ideal $I\subseteq \mathbb{H}[q_1,\ldots, q_n]$, we  prove that  for any $(a_1,\ldots,a_n)\in%either $\mathcal V(I)$  consists only of points with commuting components or 
\mathcal V(I)$   appropriate balloons, built starting from $(a_1,\cdots,a_n),$ are contained in $\mathcal{V}(I)$. %\end{Proposition}
\noindent This is the key ingredient to show that the right ideal $\mathcal{J}(\mathcal{ V }(I))$ coincides
with the set of polynomials vanishing on $\mathcal{ V }(I)$.
\noindent Moreover, after introducing the radical  $\sqrt{I}$ of a right ideal $I$ as the intersection
%        right ideal $L$ in $\mathbb{H}[q_1,\ldots,q_n]$ is {\em
%          completely prime} if for any $P,Q \in
%        \mathbb{H}[q_1,\ldots,q_n]$ such that $P*Q \in L$ and
%        $P*L\subseteq L$ we have that $P\in L$ or $Q\in L$, we have that the 
%        {\em right radical} $\sqrt{I}$ of $I$ is the intersection of
    of    all completely prime right ideals that contain $I$ (as in \cite{israeliani-1}),
 we prove this strong version of the Hilbert Nullstellensatz in $\mathbb{H}^n$
\begin{theorem*}
Let $I$ be a right ideal in $\mathbb{H}[q_1,\ldots, q_n]$. Then
\[\mathcal{J}(\mathcal{V}(I))=\sqrt{I}.\] 	
\end{theorem*}
\noindent 
{After completing the final draft of this paper we became aware that   
the previous result was simultaneously obtained by G. Alon and E. Paran, and then published in  \cite{israeliani-3}.\\
The two proofs were obtained independently, they both rely on the geometric properties of the zero sets of slice regular polynomials. %In particular on the fact that slice regular polynomials vanishing on 
The main difference is that we use an inductive argument relying on the existence of one {balloon} in any set of the form $\mathcal V(I)$, while in \cite{israeliani-3} the proof is based on the existence of a finite family of balloons (called {\em embedded spheres} in \cite{israeliani-3}), organized in the structure of a binary tree.

\noindent Let us also mention that G. Alon, A. Chapman and E. Paran further investigate the geometry of sets of the form $\mathcal V(I)$ in \cite{israeliani-4}. %providing another proof of the same result.
% using different techniques, 
}%proved by a different technique, was simultaneously obtained in \cite{israeliani-3}.\\

 \noindent
In the last part of the present paper, we introduce the notion of {\em slice algebraic
  sets} in $\mathbb{H}^n$ and show that the family of slice algebraic
sets in $\mathbb H^n$ defines a topology on $\mathbb H^n$ which resembles
the Zariski topology in $\mathbb{C}^n$. 
Thanks to the previously established results on the %symmetry and 
geometric properties of  vanishing sets of slice regular polynomials which belong to the same right ideal $I$
in $\mathbb{H}[q_1,\ldots,q_n]$,  we prove that
 $\mathcal{V}(I)$ is a slice algebraic set.
\\
\noindent The paper is organized as follows: after recalling  in Section 2 some  background results  on ideals of polynomials in
$\mathbb{H}[q_1,\ldots, q_n]$, in Section 3 we investigate the
geometric properties of the vanishing sets of right ideals of slice regular
polynomials in $n$ variables.
In the route to prove the main results of this section, we give an explicit description of slice regular polynomials vanishing on two types of subsets with spherical symmetry.
%special classes of zero-sets.
\\
%Finally, we investigate
%Finally, in Propositions \ref{caso1} and \ref{caso2},
%%for both classes of subsets, we will see how requiring that the whole ideal vanishes on them, allows us to find new points in the vanishing locus.
%we investigate how, starting from a zero set of a right ideal, it is possible to enlarge it finding new zeros of slice regular polynomials in such ideal.\\ 
%%; in particular  different types of
%%vanishing sets are considered, primarily those which have specific
%% symmetries and are preserved under the $*$-product.
In Section 4  a proof of a strong version of
 the Hilbert Nullstellensatz in $\mathbb{H}^n$ is given.
 Finally, {\em slice algebraic
  sets} in $\mathbb{H}^n$ are introduced and proved  to define 
%%  (closed) sets in $\mathbb H^n$ 
a Zarisky--type topology in $\mathbb H^n$. 
%which resemblesthe Zariski topology in $\mathbb{C}^n$. 
\section{Introduction to quaternionic slice regular polynomials}\label{sec1}
Let $\mathbb{H}{=\mathbb{R}+i\mathbb{R}+j\mathbb{R}+k\mathbb{R}}$ denote the skew field of
quaternions and let $\mathbb{S}=\{q \in \mathbb{H} \ : \ q^2=-1\}$ be
the two dimensional sphere of quaternionic imaginary units.  Then
\[ \mathbb{H}=\bigcup_{J\in \mathbb{S}}(\mathbb{R}+\mathbb{R} J),  
%\hskip 1 cm \mathbb{R}=\bigcap_{I\in \mathbb{S}}(\mathbb{R}+\mathbb{R} I),
\]
where the ``slice'' $\mathbb{C}_J:=\mathbb{R}+\mathbb{R} J$ can be
identified with the complex plane $\mathbb{C}$ for any
$J\in\mathbb{S}$. In this way, any $q\in \mathbb{H}$ can be
  expressed as $q=x+yJ$ with $x,y \in \mathbb{R}$ and $J \in
  \mathbb{S}$. The {\it real part }of $q$ is ${\rm Re}(q)=x$ and its
         {\it imaginary part} is ${\rm Im}(q)=yJ$; the {\it conjugate}
         of $q$ is $\bar q:={\rm Re}(q)-{\rm Im}(q)$.  For any non-real
         quaternion $a\in \mathbb{H}\setminus \mathbb{R}$ we will
         denote by $J_a:=\frac{{\rm Im}(a)}{|{\rm Im}(a)|}\in
         \mathbb{S}$ and by $\mathbb{S}_a:=\{{\rm Re}(a)+J |{\rm
           Im}(a)| \ : \ J\in \mathbb{S}\}$. If $a\in \mathbb{R}$,
         then $J_a$ is any imaginary unit, and $\mathbb S_a=\{a\}.$ For any $a\in \mathbb{H}$, $C_a$ will denote the set $C_a=\{q\in \mathbb{H} \ : \ qa=aq\}.$

\noindent The central object of the present paper is the class of slice regular quaternionic polynomial functions % of the form
	$P:{\mathbb{H}^n}  \to \mathbb{H}$,
	\[	 (q_1,\ldots, q_n)\mapsto P(q_1,\ldots, q_n)=
	\sum_{ \substack{\ell_1=0,\ldots, L_1  \\
	\cdots\\ \ell_n=0,\ldots, L_n} }{q_1}^{\ell_1}\cdots {q_n}^{\ell_n}a_{\ell_1,\ldots,\ell_n} \]
	with $a_{\ell_1,\ldots,\ell_n}\in\mathbb{H}$, 
where $\deg_{{q_{\ell_j}}}P:=L_j$.\\	
\noindent
These polynomial functions are  examples of {\it slice regular functions} on $\mathbb{H}^n$. %
When considering functions of several quaternionic variables, the
definition of slice regularity %is more complicated and 
relies on
the notion of {\em stem functions}. The formulation of the theory in several quaternionic variables %of slice regular functions in several quaternionic variables 
(in the more general setting of real alternative $*-$algebras)
%of slice regular functions in several variables, 
can be found in \cite{severalvariables}.

\noindent {In one quaternionic variable (under some hypothesis on the domain) this class of functions coincides with functions such that their restriction to any {\em slice} $\mathbb C_I$ is in the kernel of corresponding Cauchy-Riemann operator \[\overline{\partial_I}:=\frac{1}{2}\left(\frac{\partial}{\partial x} +I\frac{\partial}{\partial y}\right).\]
For an updated survey on the theory in one quaternionic variable we refer to the book \cite{libroGSS}. }

%For the sake of simplicity, we will give the main definitions and statements in the case of slice regular polynomial functions
%%  that can be extended to the case of slice regular analytic functions
%%% power series
%in two quaternionic variables.
%All the notions and properties discussed in this section can be easily adapted to the several variables case.
%
{In analogy with the Splitting Lemma for slice regular functions in one quaternionic variable (see \cite[Lemma 1.3]{libroGSS}), for slice regular polynomial functions in several quaternionc variables  we
have the following
\begin{lemma}\label{splitting}%\marginpar{ $K\perp L$??}
Let $P$
	be a slice regular polynomial in $n$ variables. For   any $K\in \mathbb S$ and for any $L\in\mathbb S$ orthogonal to $K$ (with respect to the standard scalar product in $\mathbb{R}^3$), there exist two complex polynomials $F,G:\mathbb C_K^{n}\to \mathbb C_K^{n} $ such that for any $(z_1,\ldots,z_n)\in \mathbb C_K^n$
	\begin{equation}\label{split}
		P(z_1,\ldots,z_n)=F(z_1,\ldots,z_n)+G(z_1,\ldots,z_n)L.
	\end{equation}
	
\end{lemma}
\begin{proof}
	Let $P(q_1,\ldots,q_n)=\sum_{\ell_1,\ldots,\ell_n}q_1^{\ell_1}\cdots q_n^{\ell_n}a_{\ell_1,\ldots,\ell_n}$ and let $K,L \in \mathbb S$ be two orthogonal imaginary units. 
	Then $\{1, K, L, KL\}$ is a basis of $\mathbb H$ over $\mathbb R$ and hence we can write any coefficient of $P$ as
	\[a_{\ell_1,\ldots,\ell_n}=\alpha_{\ell_1,\ldots,\ell_n}+\beta_{\ell_1,\ldots,\ell_n}L\]
	with $\alpha_{\ell_1,\ldots,\ell_n},\beta_{\ell_1,\ldots,\ell_n} \in \mathbb C_K$.
	Setting
	\[F(z_1,\ldots,z_n)=\sum_{\ell_1,\ldots,\ell_n}z_1^{\ell_1}\cdots z_n^{\ell_n}\alpha_{\ell_1,\ldots,\ell_n}, \ \ \text{ and }  \ \
	G(z_1,\ldots, z_n)=\sum_{\ell_1,\ldots,\ell_n}z_1^{\ell_1}\cdots z_n^{\ell_n}\beta_{\ell_1,\ldots,\ell_n}\] we get the claim.
\end{proof}

}

%{For these classes of slice regular functions it is possible to define in a natural way {\em partial Cullen derivatives} (see  \cite{severalvariables, PV}). 
%\begin{definition}
%Let $f(q_1,q_2)=\sum_{(n,m)\in \mathbb{N}^2}q_1^nq_2^ma_{n,m}$ be a slice regular function on $\Omega_{R_1,R_2}$. Then 
%the partial Cullen derivative of $f$ with respect to $q_1$ is 
%	\[f_{q_1}(q_1,q_2)=\sum_{n\ge 1,m \ge 0}n q_1^{n-1}q_2^ma_{n,m},\]
%the partial Cullen derivative of $f$ with respect to $q_2$ is 
%\[f_{q_2}(q_1,q_2)=\sum_{n\ge 0,m \ge 1}m q_1^{n}q_2^{m-1}a_{n,m}.\]
%\end{definition}
%
%}

Slice regular polynomial functions of several variables can be endowed with an
appropriate notion of product, the so called {\em slice product}, that
will be denoted by the symbol $*$.
%used in the one-variable
%case. %We refer to \cite{severalvariables} for a precise definition,
Let us recall here how it works for
% monomials, and use the fact that we can extend it by linearity to define it for polynomials and 
slice regular polynomials in two variables.
\begin{definition}
	If $P(q_1,q_2)=\sum_{ \substack{n=0,\ldots,N_1  \\ m=0,\ldots,N_2} } q_1^nq_2^m a_{n,m}$ and
        $Q(q)=\sum_{ \substack{n=0,\ldots,L_1  \\ m=0,\ldots,L_2} }q_1^nq_2^m b_{n,m}$ are two
        slice regular polynomials,
        then the {\em $*$-product} of $P$ and $Q$ is the slice regular polynomial
        defined by
$$
P*Q (q_1,q_2):=\sum_{ \substack{n=0,\ldots,N_1+L_1  \\ m=0,\ldots,N_2+L_2}}q_1^nq_2^m\sum_{ \substack{r=0,\ldots,n  \\ s=0,\ldots,m} }a_{r,s} b_{n-r,m-s}
$$
\end{definition}
\noindent For example, if $a,b \in \mathbb{H}$, then
\begin{itemize}
	\item $q_1*q_2=q_2*q_1=q_1q_2$;
	\item $a*(q_1q_2)=(q_1q_2)*a=q_1q_2a$;
	\item $(q_1^nq_2^ma)*(q_1^rq_2^sb)=q_1^{n+r}q_2^{m+s}ab$.
\end{itemize} 	
\noindent Moreover we point out that, if $P$ or $Q$ have real coefficients, then $P*Q=Q*P$.

\noindent{In the one variable case, the $*$-product of two slice-regular polynomials $P$ and $Q$ can be expressed as follows (see \cite[Theorem 3.4]{libroGSS}):
	\begin{equation}\label{PQ}
	P* Q(q)=\left\{\begin{array}{c cl}
	0 & &\text{if $P(q)=0 $}\\
	P(q)\cdot Q(P(q)^{-1}\cdot q \cdot P(q)) & &\text{if $P(q)\neq 0 $\,,}
	\end{array}\right.
	\end{equation}
where the symbol $\cdot$ denotes the standard product in $\mathbb H$.}
	\noindent
        Notice that $P(q)^{-1}\cdot q\cdot P(q)$ belongs to the %same
        sphere $\mathbb{S}_q$. Hence each zero of $P*Q$ in
        $\mathbb{S}_q$ is given either by a zero of $P$ or by a point
        which is a conjugate of a zero of $Q$ in the same sphere.

%\begin{remark}\label{valutazione}
\noindent Observe that 
(\ref{PQ}) shows that
%despite the multiplicative structure given by the $*$ product in $\mathbb{H}[q_1,q_2]$ resembles the usual one,
the point-wise evaluation of slice regular polynomials %power series
is not a multiplicative homomorphism.
Moreover, the vanishing locus of the $*$-product of two slice regular polynomials is not in general the union of the zeros of each of the factors.

\noindent For slice regular polynomials in several variables, the situation is even more complicated; in general, the zeros of $P$ are not necessarily zeros of $P*Q$. For instance, while the polynomial $q_1-i$ vanishes on $\{i\}\times \mathbb{H}$,
    the polynomial $q_1q_2-q_1j-q_2i+k$ obtained by $*$--multiplication of  $(q_1-i)$ times $(q_2-j)$, vanishes on the pair $(i, q_2)$ if and only if  $q_2\in\mathbb{C}_i$.

\noindent A first result in the direction of describing the zero locus of slice regular polynomials in several variables, is the following (see \cite[Proposition 3.6]{Nul1}).	 
	\begin{proposition}\label{mlineare}
	Let $P$
	be a slice regular polynomial in $n$ variables and let $1\le m \le n$. Then $P$ vanishes on $\mathbb{H}^{m-1}\times\{a\}\times (C_a)^{n-m}$
	if and only if there exists $P_m\in\mathbb{H}[q_1, \ldots, q_n]$ such that 
	\[P(q_1, \ldots, q_n)=(q_m-a)*P_m(q_1, \ldots, q_n).\]	
	%(And maybe the same holds for slice regular power series...)	
\end{proposition}

%\end{remark}

%	where $T_{f^c}(q):= f(q)^{-1}q f(q)$ (see \cite{libroGSS}).
%From the previous formula it is not difficult to see that if $f$ is
%{\em slice preserving}, then $f*g=f\cdot g=g*f$ It is not difficult
%to show that $(\mathcal{SR}(\Omega), +, *)$ is a non-commutative
%ring.
	
\noindent In the one variable case, the structure of the zero set of a slice regular polynomial (more in general a slice regular function) is completely known. Besides at isolated points, slice regular polynomials can also vanish on two
dimensional spheres. As an example, the polynomial $q^2+1$ vanishes on
the entire sphere of imaginary units $\mathbb{S}$. More in general, for any $a\in \mathbb H$, the quadratic polynomial with real coefficients
\[S_a(q)=(q-a)*(q-\bar{a})=q^2-2{\rm Re}(a)q+|a|^2\]
vanishes on the sphere $\mathbb{S}_a$. The polynomial $S_a(q)$ is the symmetrized polynomial $(q-a)^s$, see \cite[Definition 1.48]{libroGSS}. 

\noindent Furthermore, zeros of slice regular polynomials in one variable, including spherical ones, cannot accumulate.
\begin{theorem}\cite[Theorem 3.13]{libroGSS}
	Let $P$ be a slice regular polynomial in one variable. If $P$ does not vanish identically, then its zero
        set consists of isolated points or isolated $2$-spheres of the
        form $x +y \mathbb{S}$ with $x,y \in \mathbb{R}$, $y \neq 0$.
\end{theorem}
\noindent 
%Slice regular polynomials can be viewed as an important subclass of slice regular functions. We refer to \cite{libroGSS} for a detailed survey on the theory of slice regular functions in one quaternionic variable. 

%{\begin{remark}
%As proven in \cite{severalvariables}, partial Cullen derivatives satisfy the Leibniz rule with respect to the $*$-product.	
%\end{remark}
%}

\medskip 

\noindent In the sequel we will denote by $\mathbb{H}[q_1,\ldots,q_n]$ the set of slice regular
polynomials in $n$ quaternionic variables. Since the $*$-product is associative but not commutative, $(\mathbb{H}[q_1,\ldots,q_n], +, *)$ is
a non-commutative ring (without zero divisors). \\
%We will deal with ideals in $\mathbb{H}[q_1,\ldots,q_n]$.
%\begin{definition}
\noindent	A subset $ I$ of $\mathbb{H}[q_1,\ldots,q_n]$, closed under addition, is called
	\begin{itemize}
		\item a {\em left ideal} if for any $P\in\mathbb{H}[q_1,\ldots,q_n]$, $P* I=\{P*Q \ : \ Q \in  I\} \subseteq  I$;
		\item a {\em right ideal} if for any $P\in\mathbb{H}[q_1,\ldots,q_n]$, $I*P=\{Q*P \ : \ Q \in I\} \subseteq  I$;
		\item a {\em two-sided ideal} if $ I$ is both a left and a right ideal. 
	\end{itemize}

\noindent We are interested in studying the set of common zeros of slice regular polynomials which belong to a right ideal in
 $\mathbb{H}[q_1,\ldots,q_n]$.

\begin{definition}
	
	\noindent  Given a right ideal $I$ in $\mathbb{H}[q_1,\ldots,q_n]$,
	we define $\mathcal{V}(I)$ to be the set of common zeros of $P\in I$, i.e.,
	if $Z_P\subset \mathbb{H}^n$ represents the zero set of a slice regular  polynomial $P\in I$, then 
	\[ \mathcal{V}(I):=\bigcap_{P\in I} Z_P.\]
	
	\noindent Furthermore, we set
	
	\[ \mathcal{V}_c(I):=\mathcal{V}(I)\cap\bigcup_{J\in \mathbb{S}}(\mathbb{C}_J)^n.\]

\end{definition}
%% {\color{blue} Cancellerei questa frase: \noindent The notation is in accordance with the one in \cite{israeliani-1}, where $\mathbb{H}^n_c$ denotes the set of quaternionic $n$-tuples with commuting components, so that 
	%% $\mathcal{V}_c(I)=\mathcal{V}(I)\cap \mathbb{H}^n_c$.}
%% 
\noindent Notice that $\mathcal V_c(I)$  is contained in $\mathcal{V}(I)$. 
As in \cite{Nul1}, given  $(a_1,\ldots,a_n)\in \mathbb{H}^n$, we set $E_{(a_1,\ldots,a_n)}:=\{ P\in\mathbb{H}[q_1,\ldots,q_n] \ : \ P(a_1,\ldots,a_n)=0\}$.
Observe that a point $(a_1,\ldots,a_n)\in \mathbb{H}^n$ belongs to $\mathcal{ V }(I)$ if
and only if $I\subseteq E_{(a_1,\ldots,a_n)}$.
\noindent  Since the set of slice regular polynomials vanishing on a given subset $Z$ of $\mathbb{H}^n$ is not in general a right ideal 
(see \cite[Proposition 3.11]{Nul1}), it becomes natural to give the following.
\begin{definition}\label{idealinuovi}
	Let $Z$ be a non-empty subset of $\mathbb{H}^n$.
	\noindent We denote by
	$\mathcal{J}(Z)$  the right ideal generated in $\mathbb{H}[q_1,\ldots,q_n]$
	by slice regular polynomials which vanish on $Z$,
	\[\mathcal{J}(Z):=\left\{\sum_{k=1}^N P_k*Q_k\ :\ P_k,Q_k\in\mathbb{H}[q_1,\ldots,q_n]\ {\rm with}\ {P_k}_{|_Z}\equiv 0\right\}.\] 
\end{definition}
\noindent Recalling Example 4.7 in \cite{Nul1}, in general 
$\mathcal{J}(Z)$ does not coincide with the set of polynomials vanishing on $Z$, but if $Z=\mathcal{ V }_c(I)$ where $I$ is a right ideal, then
\begin{equation}\label{jvci}
	\mathcal J(\mathcal V_c(I))=\{P\in \mathbb H[q_1,\ldots,q_n] : \ P_{|_{\mathcal V_c(I)}} = 0 \}.
	\end{equation}
The notion of radical of an ideal as introduced in \cite{israeliani-1}
can be defined also in $\mathbb H[q_1,\ldots,q_n]$.

%\begin{definition}

%\end{definition}

\begin{definition}
	Let $I$ be a right ideal in $\mathbb{H}[q_1,\ldots,q_n]$. The
        {\em right radical} $\sqrt{I}$ of $I$ is the intersection of
        all completely prime right ideals that contain $I$, where a
        right ideal $L$ in $\mathbb{H}[q_1,\ldots,q_n]$ is {\em
          completely prime} if for any $P,Q \in
        \mathbb{H}[q_1,\ldots,q_n]$ such that $P*Q \in L$ and
        $P*L\subseteq L$ we have that $P\in L$ or $Q\in L$.
\end{definition}

\section{Vanishing sets of slice regular polynomials}

%In \cite{Nul1} we proved the following result, which allows us to prove  
%\begin{proposition}\label{euclidean}
%	Let $M\in \mathbb{H}[q_1,\ldots,q_n]$ be a monic polynomial
%	of degree $d$ in $q_j$, with $1\le j \le n$. Then, for any $P\in
%	\mathbb{H}[q_1,\ldots,q_n]$, there exist, and are unique, polynomials
%	$Q\in \mathbb{H}[q_1,\ldots,q_n]$ and $R_0, \ldots,
%	R_{d-1}\in \mathbb{H}[q_1, \ldots,q_{j-1},q_{j+1},\ldots,q_n]$ such that
%	\[P=M*Q+\sum_{k=0}^{d-1}q_j^k*R_{k}.\]
%\end{proposition}  
%{\color{red} \noindent The following result allows us to understand how to evaluate at a given point the $*$-product of slice regular polynomials. 
%}
{In this section we identify two types of subsets of $\mathbb H^n$ (with some spherical symmetry) which we prove  in Propositions \ref{idealprod} and \ref{idealcinese} to be vanishing loci of
%slice regular polynomials of 
a right ideal in $\mathbb{H}[q_1. \ldots. q_n]$.
%	: the first one is constructed starting from a product of spheres, while the second one starting from arranged spherical sets. 
In the route to prove these results, in Proposition \ref{Scom} and in Corollary \ref{formacinese} we give an explicit description of slice regular polynomials vanishing on such subsets.\\
%Finally, we investigate
Finally, in Propositions \ref{caso1} and \ref{caso2},
%for both classes of subsets, we will see how requiring that the whole ideal vanishes on them, allows us to find new points in the vanishing locus.
we investigate how, starting from a zero set of a right ideal, it is possible to enlarge it finding new zeros of slice regular polynomials in such ideal.\\ }
Let us begin with a general fact on the $*$-product of slice regular polynomials in several variables. We will adopt the following notation:
if $\textbf{q}=(q_1,\ldots,q_n)\in \mathbb H^n$ and $\ell=(\ell_1,\ldots,\ell_n)\in \mathbb N^n$,
then 
$\textbf{q}^\ell=q_1^{\ell_1}q_2^{\ell_2}\cdots q_n^{\ell_n}$.

\begin{proposition}\label{prodstar}
	Let ${  \textbf{a}}=(a_1,\ldots,a_n)\in \mathbb H^n$ be
	 such that 
$a_la_m=a_ma_l$ for any $1\leq l,m\leq n$
	and let $P,Q \in \mathbb H[q_1,\ldots,q_n]$. Then
	\[P*Q(\textbf{a})=\left\{\begin{array}{lr}
		0 & \text {if}\ P(\textbf{a})=0\\
		P(\textbf{a})\cdot Q(P(\textbf{a})^{-1}a_1P(\textbf{a}), P(\textbf{a})^{-1}a_2P(\textbf{a}),\ldots,P(\textbf{a})^{-1}a_nP(\textbf{a})) & \text{if}\ P(\textbf{a})\neq0
	\end{array}\right.\] 	
\end{proposition}
\begin{proof}
 If
  $P(\textbf{q})=\sum_\ell \textbf{q}^\ell c_\ell$
  and $Q(\textbf{q})=\sum_\ell \textbf{q}^\ell b_\ell$,
%%	With the notation of Proposition \ref{starmon}, for a point $\textbf{\em a}$ with commuting and non zero components, we get that
%%	\[(T_1(a_1,\ldots,a_n), T_2(a_2,\ldots,a_n),\ldots,T_n(a_n))=(a_1,\ldots,a_n)=\textbf{\em a}.\]
 from the definition of $*$-product, we can write
	\[P*Q(\textbf{\em a})=\sum_\ell\textbf{\em a}^\ell P(\textbf{\em a})b_\ell.\]
	If $P({\textbf{\em a}})=0$, then $P*Q(\textbf{\em a})=0$. Otherwise, we
	can also write
	\[\begin{array}{ccl}
		P*Q(\textbf{\em a})%&=&P(\textbf{\em a})\sum_\ell P(\textbf{\em a})^{-1}\textbf{\em a}^{\ell}P(\textbf{\em a})b_\ell\\
		&=&P(\textbf{\em a})\sum_\ell P(\textbf{\em a})^{-1}a_1^{\ell_1}P(\textbf{\em a})P(\textbf{\em a})^{-1}a_2^{\ell_2}P(\textbf{\em a})\cdots P(\textbf{\em a})^{-1}a_n^{\ell_n}P(\textbf{\em a})b_\ell
	\end{array}
	\]
	that is
	\[P*Q(\textbf{\em a})=P(\textbf{\em a})\cdot Q(P(\textbf{\em a})^{-1}a_1P(\textbf{\em a}), P(\textbf{\em a})^{-1}a_2P(\textbf{\em a}),\ldots,P(\textbf{\em a})^{-1}a_nP(\textbf{\em a})).\]
\end{proof}
\noindent The previous result can be regarded as a version of
Formula (\ref{PQ}) in the several variable case. Notice that it can be also generalized to the case of converging power series.\\
%\noindent {\color{red} 
%	Moreover, whenever one of the two factors has real coefficients, we have the following consequence.
%	\begin{corollary}\label{polyreale}
%	Let ${  \textbf{a}}=(a_1,\ldots,a_n)\in \mathbb H^n$ be
%such that 
%$a_la_m=a_ma_l$ for any $1\leq l,m\leq n$
%and let $P \in \mathbb R[q_1,\ldots,q_n]$. If $P({\textbf{a}})=0$, then $P*Q({\textbf{a}})=0$ for every $Q\in \mathbb H[q_1,\ldots,q_n]$.
%\end{corollary}}
\subsection{Slice regular polynomials which vanish on products of spheres}

 We start with a preliminary result concerning Euclidean division of polynomials.
	\begin{lemma}
Let $P\in \mathbb{H}[q_1,\ldots,q_n]$ and let $M,L$ be monic polynomials in $\mathbb{H}[q_1,\ldots,q_n]$ of degree $d_m$ in $q_m$ and of degree $d_\ell$ in $q_\ell$ respectively. If we divide $P$ by $M$ and then we divide the remainder of such division by $L$, we get a second remainder of degree less than $d_m$ in $q_m$ and less than $d_\ell$ in $q_\ell$.
	\end{lemma}

\begin{proof}
Applying the Euclidean division with remainder as in \cite[Proposition 3.2]{Nul1}, we can divide $P$ by the monic polynomial $M$ obtaining
$$P(q_1,\ldots,q_n)=M*P_1(q_1,\ldots,q_n)+R(q_1,\ldots,q_n),$$
with
\begin{equation}\label{resti}
	R(q_1,\ldots,q_n)=\sum_{k=1}^{d_{m}-1}q_m^k*R_k(q_1,\ldots,q_{m-1},q_{m+1},\ldots,q_n),\quad  \text{with}  \quad \deg_{q_m}R<d_m.
\end{equation} 
Similarly, we can write $R=L*Q+T$ with
$$T(q_1,\ldots,q_n)=\sum_{k=1}^{d_{\ell}-1}q_\ell^k*T_k(q_1,\ldots,q_{\ell-1},q_{\ell+1},\ldots,q_n), \quad \text{and} \deg_{q_\ell}T<d_\ell.$$
If divide each $R_k$ in \eqref{resti} by $L$, we can write $R_k=L*Q_k + T'_k$ with $\deg_{q_\ell}T'_k<d_\ell$, obtaining
\[R=\sum_{k=1}^{d_{m}-1}q_m^k*R_k=\sum_{k=1}^{d_{m}-1}q_m^k*(L*Q_k+T_k')=L*\sum_{k=1}^{d_{m}-1}q_m^k*Q_k+\sum_{k=1}^{d_{m}-1}q_m^k*T'_k.\]
Thanks to the uniqueness statement in Proposition 3.2 in \cite{Nul1}, we get that
$$T(q_1,\ldots,q_n)=\sum_{k=1}^{d_{m}-1}q_m^k*T'_k(q_1,\ldots,q_{m-1},q_{m+1},\ldots,q_n),$$
and hence we conclude that
%Hence 
$\deg_{q_m} T < d_m$.
% and $\deg_{q_\ell} T < d_\ell$.
\end{proof}	
	
\noindent Using the previous lemma, it is not difficult to prove the following.
	\begin{lemma}\label{gradouno}
		If $P\in \mathbb{H}[q_1,\ldots,q_n]$ has degree at most one in $q_1,\ldots,q_n$ and vanishes on a set of the type $\mathbb{S}_{a_1}\times \ldots\times \mathbb{S}_{a_n}$ with $a_\ell \notin \mathbb R$ for every $\ell$, then  $P$ is identically zero.
	\end{lemma}
%\begin{proof}
%Since $\deg_{q_1}P\le 1$, we can write $P(q_1,\ldots,q_n)=q_1P_1(q_2,\ldots,q_n)+P_2(q_2,\ldots,q_n)$ 
%\end{proof}
\begin{proof}
For any $(\hat a_2,\ldots,\hat a_n) \in \mathbb{S}_{a_2}\times \ldots\times \mathbb{S}_{a_n}$ The polynomial $\hat P(q_1)=P(q_1,\hat a_2,\ldots,\hat a_n)$ is slice regular and vanishes on the entire sphere $\mathbb{S}_{a_1}$. 
Since $\deg_{q_1}\hat P \le 1$, $\hat P$ is identically zero. Thus, if
$P(q_1,\ldots,q_n)=q_1P_1(q_2,\ldots,q_n)+P_2(q_2,\ldots,q_n)$, then  $P_1$ and $P_2$ vanish on $\mathbb{S}_{a_2}\times \ldots\times \mathbb{S}_{a_n}$.
Since $P_1$ and $P_2$ have degree less or equal than $1$ in each variable, we can repeat the procedure until we are left to study the case of polynomials in two variables $q_{n-1},q_n$, of degree less or equal than one in each variable, vanishing on $\mathbb{S}_{a_{n-1}}\times \mathbb{S}_{a_n}$.
Let $Q(q_{n-1},q_n)$ be such polynomial. Then
$Q(q_{n-1},a_n)$ is a one-variable polynomial vanishing for any $q_{n-1}\in \mathbb S_{a_{n-1}}$. The fact that $\deg_{q_{n-1}}Q\le 1$ implies that $Q(q_{n-1},q_n)$ equals zero on $\mathbb{H}\times \{a_n\}$. 
The same holds if we restrict to $q_n=\bar a_n$. \\
Recalling Proposition \ref{mlineare}, we can write
\[Q(q_{n-1},q_n)=(q_n-a_n)*(q_{n-1}\alpha +\beta)=q_{n-1}q_n\alpha-q_{n-1}a_n\alpha+q_n\beta-a_n\beta \text{ \ for some $\alpha,\beta \in \mathbb H$}.\]
 Hence 
\[0\equiv Q(q_{n-1},\bar a_n)=q_{n-1}\bar a_n\alpha-q_{n-1}a_n\alpha+\bar a_n\beta-a_n\beta=q_{n-1}(\bar a_n-a_n)\alpha+(\bar a_n-a_n)\beta.\]
By the Identity Principle for polynomials, $\alpha$ and $\beta$ must be zero.
%For $q_{n-1}=0$ we get $\beta=0$, for $q_{n-1}=1$ we get that also $\alpha=0$.
Therefore $Q$ is identically zero. Proceeding backwards we deduce that the same holds for $P$. 

\end{proof}

\noindent Let us now characterize polynomials vanishing on subsets of the form 	$\mathbb{S}_{a_1}\times \ldots\times \mathbb{S}_{a_k}\times \mathbb H^{n-k}$.
	
{\begin{proposition}\label{prodottisfere}
		Let $1\le k \le n$. A slice regular polynomial $P\in
		\mathbb{H}[q_1,\ldots,q_n]$ vanishes at		
 $\mathbb{S}_{a_1}\times \ldots\times \mathbb{S}_{a_k}\times \mathbb{H}^{n-k}$, with $a_l \notin \mathbb R$ for every $1\le l \le k$, if and only if $P$ is of the form
	\[			P(q_1,\ldots,q_n)=\sum_{l=1}^k S_{a_l}*P_l(q_1, \ldots,q_n),\]
	%$\mathbb{S}_{a_1}\times \ldots\times \mathbb{S}_{a_n}\subset \mathbb{H}^n$ {\color{red} with $a_l \notin \mathbb R$ for every $l$}
%		if and only if there exist
%		$P_l\in \mathbb H[q_1,\ldots,q_n]$ for any $l=1,\ldots,n$ such that
%		\begin{equation}
%			P(q_1,\ldots,q_n)=\sum_{l=1}^n S_{a_l}(q_l)*P_l(q_1, \ldots,q_n).
%		\end{equation}	
		where $S_{a_l}(q_l)=q_l^2-2 {\rm Re}(a_l)q_l+|a_l|^2$.% {\color{blue} and $S_{a_l}(q_l)=q_l -a_l$ if $a_l\in \mathbb R$}.
		%(And maybe the same holds for slice regular power series...)	

%{\color{green}	
%\[			P(q_1,\ldots,q_n)=\sum_{l=1}^k M_{a_l}(q_l)*P_l(q_1, \ldots,q_n),\]
	%$\mathbb{S}_{a_1}\times \ldots\times \mathbb{S}_{a_n}\subset \mathbb{H}^n$ {\color{red} with $a_l \notin \mathbb R$ for every $l$}
%		if and only if there exist
%		$P_l\in \mathbb H[q_1,\ldots,q_n]$ for any $l=1,\ldots,n$ such that
%		\begin{equation}
%			P(q_1,\ldots,q_n)=\sum_{l=1}^n S_{a_l}(q_l)*P_l(q_1, \ldots,q_n).
%		\end{equation}	
%		where $M_{a_l}(q_l)=q_l^2-2 Re(a_l)q_l+|a_l|^2$ if  $a_l\notin \mathbb R$
% and $M_{a_l}(q_l)=q_l -a_l$ if $a_l\in \mathbb R$.}

	\end{proposition}

{	\begin{proof}
		Let us divide $P(q_1,\ldots, q_n)$ by $S_{a_1}(q_1):=(q_1^2-2{\rm Re} (a_1)q_1+|a_1|^2)$,
		the reminder $R_1(q_1,\ldots, q_n)$
		is of degree at most 1 in $q_1$; furthermore, if one divides
		$R_1(q_1,\ldots, q_n)$ by $S_{a_2}(q_2):=(q_2^2-2{\rm Re} (a_2)q_2+|a_2|^2)$, one then writes
		\[ P(q_1,\ldots, q_n)=S_{a_1}*P_1(q_1,\ldots, q_n)+S_{a_2} *P_2(q_1,\ldots, q_n)+R_2 (q_1,\ldots, q_n)\]
		with $R_2$ a slice regular polynomial of degree at most 1 in $q_1$ and in $q_2$.
		Iterating this procedure $k$ times, one obtains
		%Hence, after repeating  similar divisions of the rests  $R_k$
		%by $S_{a_k}$, one finally obtains
		\begin{eqnarray*}  P(q_1,\ldots, q_n)&=&S_{a_1}* P_1(q_1,\ldots, q_n)+S_{a_2} *P_2(q_1,\ldots, q_n)+
			\ldots 
			\\  &\ldots& +S_{a_k}* P_k(q_1,\ldots, q_n)+R (q_1,\ldots,
			q_n)\end{eqnarray*} with $R$ a polynomial of degree at most 1 in
		$q_1,\ q_2,\ldots,q_k$.  Suppose that $P$ vanishes at
		$\mathbb{S}_{a_1}\times\cdots\times \mathbb{S}_{a_k}\times \mathbb H^{n-k}$. Then, for any fixed choice of the last $n-k$ entries $(a_{k+1},\ldots,a_n) \in \mathbb H^{n-k}$, applying Lemma \ref{gradouno}, 
		\[(q_1,\ldots,q_k) \mapsto R(q_1,\ldots,q_k,a_{k+1},\ldots,a_n) \] vanishes identically.
		Since $(a_{k+1},\ldots,a_n)$ are chosen arbitrarily, we necessarily have $R\equiv 0$. \\ Conversely if $P$ is of
		the form $P(q_1,\ldots,q_n)=\sum\limits_{l=1}^k S_{a_l}*P_l(q_1,
		\ldots,q_n)$, recalling that $S_{a_l}(q_l)$ has real coefficients for any $l$,
                %%thanks to Proposition \ref{starmon},
                we have that
                 \begin{equation}\label{polyreali}
                 	S_{a_l}*\sum_{(k_1,\ldots,
			k_n)\in K} q_1^{k_1}q_2^{k_2}\cdots q_n^{k_n} a_{k_1,\ldots, k_n}
		=\sum_{(k_1,\ldots, k_n)\in K} q_1^{k_1}q_2^{k_2}\cdots
		S_{a_{k_l}}(q_l)q_l^{k_l} \ldots q_n^{k_n} a_{k_1,\ldots, k_n}
		\end{equation}
		 and hence $P$ vanishes on  $\mathbb{S}_{a_1}\times \ldots\times \mathbb{S}_{a_k} \times \mathbb H^{n-k}$.
	\end{proof}
}

}

%\begin{remark}\label{SxH}
%	Notice that the previous result can be easily generalized for polynomials vanishing on $\mathbb{S}_{a_1}\times \ldots\times \mathbb{S}_{a_k}\times \mathbb{H}^{n-k}$. More precisely it can be proven that a slice regular polynomial $P$ vanishes on $\mathbb{S}_{a_1}\times \ldots\times \mathbb{S}_{a_k}\times \mathbb{H}^{n-k}$ if and only if $P$ is of the form
%	\[			P(q_1,\ldots,q_n)=\sum_{l=1}^k S_{a_l}(q_l)*P_l(q_1, \ldots,q_n).\]
%\end{remark}	
\noindent	Let us now recall {Proposition 3.10 in \cite{Nul1}}.

	\begin{proposition}\label{zeri}
		A slice regular polynomial  $P\in
		\mathbb{H}[q_1,\ldots,q_n]$ vanishes at $(a_1,\ldots,a_n)\in \mathbb{H}^n$ if and only if there exist
		$P_k\in \mathbb H[q_1,\ldots,q_k]$ for any $k=1,\ldots,n$ such that
		\begin{equation}\label{zeroab}
			P(q_1,\ldots,q_n)=\sum_{k=1}^n(q_k-a_k)*P_k(q_1, \ldots,q_k).
		\end{equation}	
		%(And maybe the same holds for slice regular power series...)	
	\end{proposition}
	
\noindent Combining the two previous propositions, we obtain the following characterization of slice regular polynomials vanishing on a subset of the form $\mathbb{S}_{a_1}\times \ldots\times \mathbb{S}_{a_k}\times\{a_{k+1}\}\times \ldots\times \{a_{n}\}$.
{\begin{proposition}\label{Scom}
		A slice regular polynomial $P\in \mathbb{H}[q_1,\ldots,q_n]$ vanishes on
		$\mathbb{S}_{a_1}\times \ldots\times \mathbb{S}_{a_k}\times\{a_{k+1}\}\times \ldots\times \{a_{n}\}\subset
		\mathbb{H}^n$,  with $a_l \notin \mathbb R$ for every $1\le l\le k$, if and only if there
		exist $P_l\in \mathbb H[q_1,\ldots,q_n]$ for any $l=1,\ldots, k$
		and $P_l\in \mathbb H[q_1,\ldots,q_l]$ for any $l=k+1,\ldots, n$ such that
		\begin{equation}\label{spezzo1}
			P(q_1,\ldots,q_n)=\sum_{l=1}^k S_{a_l}*P_l(q_1, \ldots,{q_n})+\sum_{l=k+1}^n (q_l-a_l)*P_l(q_1, \ldots,q_l).
		\end{equation}	
		If moreover $a_la_m=a_ma_l$ for any $l,m> k$, then $P$ vanishes on $\mathbb{S}_{a_1}\times \ldots\times \mathbb{S}_{a_k}\times\{a_{k+1}\}\times \ldots\times \{a_{n}\}\subset
		\mathbb{H}^n$ if and only if $P$
		is of the form
		\begin{equation}\label{spezzo}
			P(q_1,\ldots,q_n)=\sum_{l=1}^k S_{a_l}*P_l(q_1, \ldots,q_n)+\sum_{l=k+1}^n (q_l-a_l)*P_l(q_1, \ldots,q_n),
		\end{equation}	
		with $P_l\in \mathbb H[q_1,\ldots,q_n]$ for any $l=1,\ldots,n$. 
%		 can be
	\end{proposition}
	\begin{proof}
	Let us prove the first part of the statement. If $P$ is as in
		\eqref{spezzo1}, then, thanks to Proposition \ref{prodottisfere} and Proposition \ref{zeri}, $P$ vanishes on $\mathbb{S}_{a_1}\times
		\ldots\times \mathbb{S}_{a_k}\times\{a_{k+1}\}\times \ldots\times
		\{a_{n}\}$. 
				
		Conversely, assume $P$ vanishes on $\mathbb{S}_{a_1}\times \ldots\times \mathbb{S}_{a_k}\times\{a_{k+1}\}\times \ldots\times \{a_{n}\}$; in particular
		$P$ vanishes at $(a_1,\ldots, a_n)$, therefore, from   Proposition \ref{zeri}, we have
		\begin{equation*}\label{spezzo2}
			P(q_1,\ldots,q_n)=\sum_{l=1}^k(q_l-a_l)*P_l(q_1, \ldots,q_l)+\sum_{l=k+1}^n(q_l-a_l)*P_l(q_1, \ldots,q_l);
		\end{equation*}	
Consider 
\[\hat P(q_1,\ldots,q_k):= P(q_1,\ldots,q_k,a_{k+1},\ldots,a_n)=\sum_{l=1}^k(q_l-a_l)*P_l(q_1, \ldots,q_l)\] 
which vanishes on 
$\mathbb{S}_{a_1}\times
\ldots\times \mathbb{S}_{a_k}$ and thus, applying Proposition \ref{prodottisfere}, 
\[\hat P(q_1,\ldots,q_l)=\sum_{l=1}^k S_{a_l}*\hat P_l(q_1, \ldots,q_l),\]
which leads to the expression \eqref{spezzo1}.
Since \eqref{spezzo1} is a particular instance of \eqref{spezzo}, to conclude the proof of the second part of the statement, it suffices to show that if $P$ is as in \eqref{spezzo}, with $a_{k+1},\ldots,a_n$ on a the same slice, then $P$ vanishes on $\mathbb{S}_{a_1}\times
\ldots\times \mathbb{S}_{a_k}\times\{a_{k+1}\}\times \ldots\times
\{a_{n}\}$.
But this is guaranteed by 
Proposition \ref{mlineare} and Proposition \ref{prodottisfere}.

%  {\color{red} with
%		$a_la_m=a_ma_l$ for $l,m> k$}, the second summand of Equation
%		(\ref{spezzo1}) vanishes. Thus, after applying Proposition
%		\ref{prodottisfere} to the first summand, we get the claim.
	\end{proof}}

	\noindent In accordance with the notations introduced in \cite{Nul1}, for any subset $U$ in $\mathbb{H}^n$, we denote by $E_U$ the set of polynomials vanishing on $U$. 

\noindent In particular,  
$E_{\mathbb{S}_{a_1}\times \ldots\times  \mathbb{S}_{a_k}\times \{a_{k+1}\} \times \ldots \times \{a_{n}\}}$
denotes the set of polynomials vanishing on $\mathbb{S}_{a_1}\times
\ldots\times \mathbb{S}_{a_k}\times\{a_{k+1}\}\times \ldots\times
\{a_{n}\}$ (with $1\leq k\leq n$).
	
\noindent A direct consequence of Formula \eqref{spezzo} in Proposition \ref{Scom} is the following.	
	
	\begin{proposition}\label{idealprod}
Let $(a_1,\ldots,a_n)\in \mathbb H^n$ and let $k$, with $1\leq k\leq n$, be such that {$a_l \notin \mathbb R$ for every $1\le l\le k$,} and
$a_la_m=a_ma_l$ for any $l,m>k$.
Then, $E_{\mathbb{S}_{a_1}\times \ldots\times \mathbb{S}_{a_k}\times\{a_{k+1}\}\times \ldots\times \{a_{n}\}}$
is a right ideal contained in $E_{(a_1,\ldots, a_n)}$ .
	\end{proposition}

\noindent{Given a right ideal $I\subset E_{(a_1, a_2,\ldots, a_n)}$, we want to explore to what extent it is possible to enlarge the singleton  
$\{(a_1, a_2,\ldots,a_n)\}$ to a set $U$, with some spherical symmetry,  in  such a way  that $I\subset E_U$.
The idea is to proceed by subsequent steps, starting with the following lemma. }
{\begin{lemma}\label{sfera1}
	If $I$ is a right ideal contained in $E_{(a_1,\ldots,a_n)}$ with $a_1a_2\neq a_2a_1$, then $I\subset E_{\mathbb S_{a_1}\times \{a_2\}\times\cdots\times \{a_n\}}$.
\end{lemma}
\begin{proof}
	Consider $P\in I$. 
By direct computation, if 
\[P(q_1,\ldots,q_n)=\sum_{(k_1,\ldots,
	k_n)\in K} q_1^{k_1}q_2^{k_2}\cdots q_n^{k_n} a_{k_1,\ldots, k_n},\]
then	
%	=\sum_{(k_1,\ldots, k_n)\in K} q_1^{k_1}q_2^{k_2}\cdots
%			Q_{a_{\ell}}(q_l)q_l^{k_l} \ldots q_n^{k_n} a_{k_1,\ldots, k_n}\] 
%{\color{red} Thanks to Proposition \ref{starmon}}
\begin{equation*}
	\begin{aligned}
		q_{2}*P(q_1,\ldots,q_n)&=\sum_{(k_1,\ldots,
			k_n)\in K} q_1^{k_1}q_2^{k_2+1}q_3^{k_3}\cdots q_n^{k_n} a_{k_1,\ldots, k_n}=q_{2}\cdot
		P(q_{2}^{-1}q_1q_{2},q_2,\ldots q_n)
	\end{aligned}
\end{equation*}
when $q_{2}\neq 0$. 	
Since $P*q_2=q_2*P \in I$,		%$P*q_2 = q_2P(q_2^{-1}q_1q_2, q_2,\ldots,q_n)$
		it vanishes at $(a_1,\ldots,a_n)$. Hence, if we consider $\hat P(q_1)=P(q_1,a_2,\ldots,a_n)$ we have that $\hat P(a_1)=0$ and $\hat P(a_2^{-1}a_1a_2)=0$. Since $\hat P$ is a slice regular polynomial in one quaternionic variable, thanks to Theorem 3.1 in \cite{libroGSS},  $\hat P$ vanishes on $\mathbb S_{a_1}$. Therefore $P \in E_{\mathbb S_{a_1}\times \{a_2\}\times\cdots\times \{a_n\}}$. 
\end{proof}
}\noindent To iterate this procedure, we need the following result.
\begin{lemma}\label{intsfere}
{Let $a_l \notin \mathbb R$ for every $1\le l\le n$.}	If $P\in E_{\mathbb{S}_{a_1}\times\cdots \times\mathbb S_{a_{n-1}}\times\{a_n\}}\cap E_{\mathbb{S}_{a_1}\times\cdots \times\mathbb S_{a_{n-1}}\times\{\widetilde a_n\}}$ with $\widetilde a_n \in \mathbb S_{a_n}$, then 
	$P\in E_{\mathbb{S}_{a_1}\times\cdots \times\mathbb S_{a_n}}$.
\end{lemma}
\begin{proof}
From Proposition \ref{Scom}, $P$ has the form
	\[P(q_1,\ldots,q_n)=\sum_{k=1}^{n-1}S_{a_k}*P_k(q_1,\ldots,q_k)+R_n(q_1,\ldots,q_n),\] 
%	with $\deg_{q_\ell}R_n<2$ for $\ell=1,\ldots n-1$. 
	where $R_n$ can be written in two different ways:
	\[R_n(q_1,\ldots q_n)=(q_n-a_n)*P_n(q_1,\ldots,q_n)=(q_n-\tilde a_n)*\tilde P_n(q_1,\ldots,q_n).\] 
	Thanks to Proposition \ref{mlineare}, $R_n$ vanishes on $\mathbb H^{n-1} \times \{a_n\}$ and on $\mathbb H^{n-1} \times \{\tilde a_n\}.$
	If we write
\[R_n(q_1,\ldots,q_n)=\sum_{\ell_k \in \{0,1\}}q_1^{\ell_1}\cdots q_{n-1}^{\ell_{n-1}}C_{\ell_1,\ldots,\ell_{n-1}}(q_n),\] we get then that 
$C_{\ell_1,\ldots,\ell_{n-1}}(a_n)= C_{\ell_1,\ldots,\ell_{n-1}}(\tilde a_n)=0$ for any choice of $\ell_1,\ldots,\ell_{n-1}$. 
Since each $C_{\ell_1,\ldots,\ell_{n-1}}$ is a slice regular polynomial in one quaternionic variable, thanks to Theorem 3.1 in \cite{libroGSS} we get that any $C_{\ell_1,\ldots,\ell_{n-1}}(q_n)\equiv 0$ on $\mathbb S_{a_n}$. That is, for any choice of $\ell_1,\ldots,\ell_{n-1}$, 
\[C_{\ell_1,\ldots,\ell_{n-1}}(q_n)=S_{a_n}*\tilde C_{\ell_1,\ldots,\ell_{n-1}}(q_n).\]
Using the fact that $S_{a_n}$ has real coefficients, we can then write
\[R_n(q_1,\ldots,q_n)=S_{a_n}*\sum_{\ell_k \in \{0,1\}}q_1^{\ell_1}\cdots q_{n-1}^{\ell_{n-1}}\tilde C_{\ell_1,\ldots,\ell_{n-1}}(q_n)=S_{a_n}*P_n(q_1,\ldots,q_n),\]
proving that $P$ vanishes on $\mathbb{S}_{a_1}\times\cdots \times\mathbb S_{a_n}$.
\end{proof}

\noindent Combining Lemmas \ref{sfera1} and \ref{intsfere}, we obtain the following.

\begin{proposition}\label{caso1}
	If $I$ is a right ideal contained in $E_{(a_1,\ldots,a_n)}$ with $a_\ell a_{\ell+1}\neq a_{\ell+1}a_\ell$ for any $\ell= 1,\ldots,k$, then $I\subset E_{\mathbb S_{a_1}\times \cdots \times \mathbb S_{a_k}\times\{a_{k+1}\}\cdots\times \{a_n\}}$.
\end{proposition}
\begin{proof}
	Since $a_1a_2\neq a_2a_1$, applying Lemma \ref{sfera1}, we get that $I\subset E_{\mathbb S_{a_1}\times \{a_{2}\}\cdots\times \{a_n\}}$.
	Let $P\in I$ and consider $P*q_3$ which still belongs to $I$. Then, { if $q_3\neq 0$},
	\[P*q_3 = q_3P(q_3^{-1}q_1q_3, q_3^{-1}q_2q_3,q_3\ldots,q_n)\] 
	vanishes on $\mathbb S_{a_1}\times\{a_2\}\times\cdots\times\{a_n\}$. %Let $\tilde a_1\in \mathbb S_{a_1}\cap C_{a_3}$ and 
Let $\tilde a_2=a_3^{-1}a_2a_3\neq a_2$.
Then for any $q_1\in \mathbb S_{a_1}$, $P(q_1,\tilde a_2,a_3,\ldots,a_n)=0$ and $P(q_1, a_2,a_3,\ldots,a_n)=0$.
Then, thanks to Lemma \ref{intsfere}, $\hat P(q_1,q_2)=P(q_1,q_2,a_3,\ldots,a_n)$   vanishes on $\mathbb S_{a_1} \times \mathbb S_{a_2}$, and hence $P \in E_{\mathbb S_{a_1}\times \mathbb S_{a_2}\times\{a_{3}\}\cdots\times \{a_n\}}$.
We can repeat this argument untill we find the first commuting pair of components $a_{k+1},a_{k+2}$, showing that 
$P\in E_{\mathbb S_{a_1}\times \cdots \times \mathbb S_{a_k}\times\{a_{k+1}\}\times\cdots\times \{a_n\}}$.
\end{proof}

\subsection{Slice regular polynomials which vanish on spherical sets}

We begin with defining the main geometrical objects we will deal with in this subsection.

\begin{definition}
For any $(a_1,\ldots,a_n)\in \mathbb {H}^n$ we denote by $\mathbb S_{(a_1,\ldots,a_n)}$ the set obtained {\em rotating} simultaneously every coordinate of the point $(a_1,\ldots,a_n)$, that is
\[\mathbb S_{(a_1,\ldots,a_n)}=\{(g^{-1}a_1g, \ldots, g^{-1}a_ng) : g \in \mathbb{H}\setminus\{0\} \}.\]
 Any set of this form will be called a {\em spherical set}. In particular, if $(a_1,\ldots,a_n)$ has commuting components, i.e.  $a_la_m=a_ma_l$ for any  $l,m$, then the spherical set will be said an {\em arranged} spherical set. 

\end{definition}	
%% 	If we start from a point $(a_1,\ldots,a_n)$ with commuting
   %%         components,
\noindent Notice that each point in an arranged spherical set %% $\mathbb S_{(a_1,\ldots,a_n)}$,
 has commuting components. In particular, if each $a_j\in\mathbb{R}, \ j=1, \ldots, n$, then
 $\mathbb S_{(a_1,\ldots,a_n)}=(a_1,\ldots,a_n)$. {If some of the $a_l$ are reals, we obtain a \textquotedblleft pinched" spherical set}.
%On the other hand, if  $a_k\in\mathbb{R}$, then
% $\mathbb S_{(a_1,\ldots,a_n)}= \mathbb S_{(a_1,\ldots,a_{k-1})}\times\{a_k\}\times
 %\mathbb S_{(a_{k+1},\ldots,a_n)}$.
%\begin{remark}\label{mangia}
	Observe that,  for any $(a_1,\ldots,a_k)\in \mathbb H^k$, the set $\mathbb S_{(a_1,\ldots,a_k)}$ is a subset of $\mathbb S_{a_1}\times\cdots \times \mathbb S_{a_k}$. Therefore $E_{\mathbb S_{a_1}\times \cdots \times \mathbb S_{a_k}\times \{a_{k+1}\} \times \cdots\times \{  a_{n}\}}$ is contained in $E_{\mathbb S_{(a_1,\cdots,a_k)}\times \{a_{k+1}\} \times \cdots\times \{  a_{n}\}}$.
%\end{remark}
%In the sequel, any set of the form $\mathbb S_{(a_1,\cdots,a_k)}\times \{a_{k+1}\} \times \cdots\times \{  a_{n}\}$ will be called a {\em balloon}.

\noindent In $\mathbb{H}^n$, arranged spherical sets play the role of spherical zeros of  slice regular polynomials in one quaternionic variable; a first result in this sense is the following

\begin{proposition}\label{zerisferici}
	A slice regular polynomial with real coefficients $R\in
        \mathbb R[q_1,\ldots,q_n]$ vanishes on a point
        $(a_1,\ldots,a_n)\in \mathbb H^n$, then it vanishes on the
        entire spherical set $\mathbb S_{(a_1,\ldots,a_n)}$.
\end{proposition}
\begin{proof}
	Let $R(q_1,\ldots,q_n)=\sum q_1^{\ell_1}\cdots q_n^{\ell_n} r_{\ell_1,\ldots, \ell_n}$ with $r_{\ell_1,\ldots, \ell_n}\in \mathbb R$.
	Then, for any $g\in\mathbb H\setminus\{0\}$, one has 
			\begin{equation}
	\begin{aligned}
					R(g^{-1}q_1g,\ldots,g^{-1}q_ng)&=\sum g^{-1}q_1^{\ell_1}g\cdots g^{-1}q_n^{\ell_n} g \, r_{\ell_1,\ldots, \ell_n}\\
					&=g^{-1}R(q_1,\ldots,q_n)g.
		 \end{aligned}
		\end{equation}
\end{proof}

\noindent Another geometrical object with spherical symmetry central for our investigation is the following.
{\begin{definition}
		Any set of the form 
		\[\mathbb S_{(a_1,\ldots,a_t)}\times \{a_{t+1}\}\times \cdots \times \{a_n\}\] 
		where 
		% $a_1,\ldots,a_n \in \mathbb H$ are such that 
		$\mathbb S_{(a_1,\ldots,a_t)}$ is an arranged spherical set, and $a_{l}a_m=a_m a_{l}$ for any $l, m > t$ will be called a {\em balloon}. 
\end{definition}}
\noindent We want to describe the set of slice regular polynomials which vanish on arranged spherical sets and on balloons.

%\begin{proposition}
%	If $I$ is an ideal in $E_{\mathbb S_{(a_1,\cdots,a_k)}\times \{a_{k+1}\} \times\{ a_{k+2}\}}$ with $a_{k+1}a_{k+2}\neq a_{k+2}a_{k+1}$, then 
%	$I\subset  E_{\mathbb S_{(a_1,\cdots,a_k, a_{k+1})} \times \{a_{k+2}\}}$.
%\end{proposition}
%\begin{proof}
%Let $P\in I$. Then $P*q_{k+1} \in I$ and $P*q_{k+2} \in I$, and hence
%\[q_{k+1}P(q_{k+1}^{-1}q_1q_{k+1},\ldots,q_{k+1}^{-1}q_kq_{k+1}, q_{k+1},q_{k+2})\]
%and
%\[q_{k+2}P(q_{k+2}^{-1}q_1q_{k+2},\ldots,q_{k+2}^{-1}q_kq_{k+2}, q_{k+2}^{-1}q_{k+1}q_{k+2},q_{k+2})\] 
%vanishes on $\mathbb S_{(a_1,\cdots,a_k)}\times a_{k+1} \times a_{k+2}$. 
%Let $\tilde a_i \in \mathbb S_{a_i} $ such that $\tilde{ a_{i}} a_{k+1}= a_{k+1}\tilde{ a_{i}} $ and   $\hat a_i \in \mathbb S_{a_i} $ such that $\hat{ a_i} =a_{k+2}^{-1}\tilde{ a_i}a_{k+2}$. 
%Therefore $P$ vanishes  at $(\tilde a_1, \ldots, \tilde a_k,a_{k+1},a_{k+2})$ and $(\hat a_1, \ldots, \hat a_k,a_{k+2}^{-1}a_{k+1}a_{k+2},a_{k+2})$.
%Consider now $\hat P=P(q_1,\ldots,q_{k+1},a_{k+2})$. $\hat P$ vanishes on two points on the same chinese sphere. Thanks to Proposition \ref{RepFor} $\hat P \in E_{\mathbb S_{(a_1,\ldots,a_{k+1})}},$ and hence $P$ vanishes on  $\mathbb S_{(a_1,\ldots,a_{k+1})}\times a_{k+2}$. 
%\end{proof}	
%With the same argument, the previous result generalizes as follows.
%\noindent 
%
%In particular, if $\mathbb{S}_{(a_1,\ldots,a_n)}$ is an arranged spherical set, then 
% $E_{\mathbb {S}_{(a_1,\ldots,a_n)}}$ is a right ideal in $\mathbb{H}[q_1,\ldots, q_n]$ whose elements are described by the following 
{\begin{proposition}
 		Let $\mathbb{S}_{(a_1,\cdots,a_n)}\subset\mathbb{H}^n$ be an arranged spherical set and suppose $a_\ell \not \in \mathbb{R}$ for any $\ell$. A slice regular polynomial $P\in\mathbb{H}[q_1,\ldots,q_n]$ vanishes on  $\mathbb{S}_{(a_1,a_2,\cdots,a_n)}$ if and only if $P$ is of the form
 		\[Q_1*P_1(q_1,\ldots,q_n)+Q_2*P_2(q_2,\ldots,q_n)+\cdots+Q_{n-1}*P_{n-1}(q_{n-1},q_n)+S_{a_n}*P_n(q_n)\]
 		where
 		$Q_\ell$ are the polynomials 
 		\begin{equation}\label{Ql}
 			Q_{\ell}(q_{\ell},q_{\ell+1})=q_{\ell}+q_{\ell +1} \gamma_{\ell} +\delta_{\ell}
 		\end{equation}
 		with real coefficients
 		\[\gamma_{\ell}= -\dfrac{|{\rm Im}  (a_{\ell})|}{|{\rm Im} ( a_{\ell +1})|}, \quad \text{and} \quad \delta_{\ell}=-{\rm Re}  (a_{\ell})+\dfrac{{\rm Re}  (a_{\ell+1})|{\rm Im}  (a_{\ell})|}{|{\rm Im}  (a_{\ell +1})|}, \]
 	for any $\ell=1,\ldots,n-1$, if ${\rm Im}(a_{\ell}){\rm Im }(a_{\ell+1})>0$; 
 	while
 		\[\gamma_{\ell}=\dfrac{|{\rm Im}  (a_{\ell})|}{|{\rm Im} ( a_{\ell +1})|}, \quad \text{and} \quad \delta_{\ell}=-{\rm Re}  (a_{\ell})-\dfrac{{\rm Re}  (a_{\ell+1})|{\rm Im}  (a_{\ell})|}{|{\rm Im}  (a_{\ell +1})|}, \]
	for any $\ell=1,\ldots,n-1$, if ${\rm Im}(a_{\ell}){\rm Im }(a_{\ell+1})<0$. 
% 	 where we take the $+$ in $\gamma_\ell$ and the $-$ in $\delta_\ell$ if ${\rm Im}(a_{\ell}){\rm Im }(a_{\ell+1})>0$, and the opposite otherwise.% he signs are determined by the sign of the product .
\end{proposition}
} 
\begin{proof}
By direct computation, each $Q_\ell$ vanishes on $(a_\ell,a_{\ell+1})$ and hence, considered as a polynomial in $\mathbb H[q_1,\ldots,q_n]$, also on $(a_1,\ldots,a_n)$. Since they have real coefficients, they all vanish on $\mathbb{S}_{(a_1,a_2,\cdots,a_n)}$ (see Proposition \ref{zerisferici}).
Consider $P\in E_{\mathbb S_{(a_1,a_2,a_3,\ldots, a_n)}}$. Performing subsequent divisions by the monic polynomials $Q_\ell$, as in \cite[Proposition 3.2]{Nul1}, we can write $P$ as
\[Q_1*P_1(q_1,\ldots,q_n)+Q_2*P_2(q_2,\ldots,q_n)+\cdots+Q_{n-1}*P_{n-1}(q_{n-1},q_n)+R(q_n).\]
Since $P,Q_1,\ldots,Q_{n-1}$ vanish on $\mathbb S_{(a_1,a_2,a_3,\ldots, a_n)}$ and $R$ depends only on the variable $q_n$, we get that $R(q_n)$ vanishes on $\mathbb S_{a_n}$. Recalling the one-variable theory \cite[Proposition 3.17]{libroGSS}, we get that $R(q_n)=(q_n-a_n)^s*P_n(q_n)=S_{a_n}*P_n(q_n)$.
\end{proof}
 %\begin{remark}
 \noindent  %and Lemma \ref{scomcom}, 
{It is easy to generalize the previous result to the case of slice regular polynomials 
which vanish on {balloons} of the form  $\mathbb S_{(a_1,\ldots,a_k)}\times \{a_{k+1}\} \times \cdots\times \{ a_{n}\}$, obtaining a characterization analogue to that given in Formula \eqref{spezzo} of Proposition \ref{Scom}.
}%, where $\mathbb S_{(a_1,\ldots,a_k)}$ is an arranged spherical set and 
% $a_{l} a_m=a_ma_{l}$ for all $ k+1\leq  l,m\leq n$.
 %\end{remark}
{\begin{corollary}\label{formacinese}
 A slice regular polynomial $P$ vanishes on the balloon $\mathbb S_{(a_1,\ldots,a_k)}\times \{a_{k+1}\} \times \cdots\times \{ a_{n}\}$  if and only if $P$ can be written as
	\begin{equation*}
		\begin{aligned}
			&\sum_{\ell=1}^{k-1}Q_\ell*P_\ell(q_\ell,\ldots,q_n)+S_{a_k}*P_k(q_k, \ldots, q_n)+\sum_{\ell=k+1}^n(q_\ell-a_\ell)*P_\ell(q_1,\ldots,q_n),		
			%	Q_1*P_1(q_1,\ldots,q_n)+Q_2*P_2(q_2,\ldots,q_n)+\cdots+Q_{k-1}*P_{k-1}(q_{k-1},\ldots, q_n)\\
		\end{aligned}
	\end{equation*}
	where the slice regular polynomials $Q_\ell$ are defined as in \eqref{Ql} and $P_\ell$ are slice regular polynomials for any $\ell$.
\end{corollary}
}

\noindent {
	As a consequence, we have that the set of polynomials vanishing on a balloon is indeed a right ideal.
	
	\begin{proposition}\label{idealcinese}
		Let $\mathbb S_{(a_1,\cdots,a_k)}\times \{a_{k+1}\}\times \cdots \times \{a_n\}$ be a balloon. Then the set	$E_{\mathbb S_{(a_1,\cdots,a_k)}\times \{a_{k+1}\}\times \cdots \times \{a_n\}}$ is a right ideal of $\mathbb{H}[q_1,\ldots,q_n]$ contained in $E_{(a_1,\ldots,a_{n})}$.
	\end{proposition}
	
	\begin{proof} 
		
		Let $P \in E_{\mathbb S_{(a_1,\cdots,a_k)}\times \{a_{k+1}\}\times \cdots \times \{a_n\}}$. We have to prove that $P*Q$ still vanishes on the balloon $\mathbb S_{(a_1,\cdots,a_k)}\times \{a_{k+1}\}\times \cdots \times \{a_n\}$ for every $Q\in \mathbb{H}[q_1,\ldots,q_n]$. Thanks to Corollary \ref{formacinese},  
		\[P(q_1,\ldots,q_n)=	\sum_{\ell=1}^{k-1}Q_\ell*P_\ell(q_\ell,\ldots,q_n)+S_{a_k}*P_k(q_k, \ldots, q_n)+\sum_{\ell=k+1}^n(q_\ell-a_\ell)*P_\ell(q_1,\ldots,q_n).\] 
		Since $Q_\ell$ and $S_{a_k}$ have real coefficients, arguing as in Equation \eqref{polyreali}, we get that
		$Q_\ell*R$ and $S_{a_k}*R$ still vanish on the balloon $\mathbb S_{(a_1,\cdots,a_k)}\times \{a_{k+1}\}\times \cdots \times \{a_n\}$ for every $R \in \mathbb H[q_1,\ldots,q_n]$ and for every $\ell$. 
		Recalling Proposition \ref{mlineare}, we conclude that $P*Q \in E_{\mathbb S_{(a_1,\cdots,a_k)}\times \{a_{k+1}\}\times \cdots \times \{a_n\}}$ for any $Q\in \mathbb H[q_1,\ldots,q_n]$.
		%%thanks to Proposition \ref{starmon},
%		we have that \[Q_{\ell}*\sum_{(k_1,\ldots,
%			k_n)\in K} q_1^{k_1}q_2^{k_2}\cdots q_n^{k_n} a_{k_1,\ldots, k_n}
%		=\sum_{(k_1,\ldots, k_n)\in K} q_1^{k_1}q_2^{k_2}\cdots
%		Q_{a_{\ell}}(q_l)q_l^{k_l} \ldots q_n^{k_n} a_{k_1,\ldots, k_n}\] and hence $P$ vanishes on  $\mathbb{S}_{a_1}\times \ldots\times \mathbb{S}_{a_k} \times \mathbb H^{n-k}$.
		\end{proof}
%	\begin{proof}
%		The proof follows by
%	\end{proof}
}
\noindent In the last part of this subsection we use the information on some common zeros of slice regular polynomials in a right ideal to describe an enlarged common vanishing set of the same polynomials.

\noindent Let us now prove the following simplified version of the Representation Formula for slice regular polynomials in several variables (see \cite[Proposition 2.12]{severalvariables}), which holds on the {\em spherical cylinder} $\bigcup\limits_{J\in \mathbb S}\mathbb C_J^n$.
%Formula di rappresentazione su sfere cinesi
\begin{proposition}\label{RepFor}
	Let $P\in \mathbb H[q_1,\ldots,q_n]$ and let $J,K \in \mathbb S$. If $w_1,\ldots,w_n \in \mathbb {C}_J$ and $z_1, \ldots, z_n$ are their ``shadows" on $\mathbb{C}_K$, that is if $w_\ell=x_\ell+y_\ell J$ with $y_\ell \ge 0$, then $z_\ell=x_\ell + y_\ell K$, then
	\begin{equation}\label{RepFor1}
		P(w_1,\ldots, w_n)=\frac{1-JK}{2}P(z_1,\ldots,z_n)+\frac{1+JK}{2}P(\bar z_1,\ldots,\bar z_n).\end{equation}   
\end{proposition}
\begin{proof}
	Let us show the result for the monomial $M(q_1,\ldots,q_n)=q_1^{\ell_1}q_2^{\ell_2}\cdots q_n^{\ell_n}a$, with $a\in \mathbb H$.	
	If $w_1,\ldots,w_n\in \mathbb C_J$, then there exist $A, B \in \mathbb{R}$ such that 
	$w_1^{\ell_1}w_2^{\ell_2}\cdots w_n^{\ell_n}=A+BJ$. 
	Moreover, since $A$ and $B$ do not depend on $J$, 
	if $z_1,\ldots,z_n$ are the "shadows" of $w_1,\ldots,w_n$ on $\mathbb C_K$, then $z_1^{\ell_1}z_2^{\ell_2}\cdots z_n^{\ell_n}=A+BK$. Therefore
	
	\begin{equation*}
		\begin{aligned}
			&\frac{1-JK}{2}M(z_1,\ldots,z_n)+\frac{1+JK}{2}M(\bar z_1,\ldots,\bar z_n)\\
			&=\frac{1-JK}{2}z_1^{\ell_1}z_2^{\ell_2}\cdots z_n^{\ell_n}a+\frac{1+JK}{2}\bar z_1^{\ell_1}\bar z_2^{\ell_2}\cdots \bar z_n^{\ell_n}a\\
			&=\frac{1-JK}{2}(A+BK)a+\frac{1+JK}{2}(A-BK)a=(A+BJ)a=M(w_1,\ldots,w_n).
		\end{aligned}
	\end{equation*}

	%$M(w_1,\ldots,w_n)=$
	
\end{proof}

\noindent A direct application of Proposition \ref{RepFor} which can be regarded as the several variable version of \cite[Theorem 3.1] {libroGSS} is the following

\begin{proposition}\label{zericinesi}
	Let $P\in \mathbb{H}[q_1,\ldots,q_n]$ and let $(a_1,\ldots,a_n)\in \mathbb{C}_L^n$ for some $L\in \mathbb{S}$.
	If $P$ vanishes on two different points of $\mathbb S_{(a_1,\ldots,a_n)}$, then $P$ vanishes on the entire spherical set $\mathbb S_{(a_1,\ldots,a_n)}$.
\end{proposition}	
\begin{proof}
	{If $P$ vanishes at two conjugated points $(b_1,\ldots,b_n), (\bar b_1,\ldots,\bar b_n) \in \mathbb C_K^n\cap \mathbb S_{(a_1,\ldots,a_n)}$ and $(c_1,\ldots,c_n)$ is another point in $\mathbb C_J^n\cap \mathbb S_{(a_1,\ldots,a_n)}$, with $J\neq K$, then using Proposition \ref{RepFor}, we can write
		\[P(c_1,\ldots,c_n)=\frac{1-JK}{2}P(b_1,\ldots,b_n)+\frac{1+JK}{2}P(\bar b_1,\ldots,\bar b_n)=0.\]}
	%\frac{1+JK}{2}P(\bar b_1,\ldots,\bar b_n).\]   	
	
	%If $P$ vanishes at two conjugated points $(b_1,\ldots,b_n)$ and $(\bar b_1,\ldots,\bar b_n)$ of $\mathbb S_{(a_1,\ldots,a_n)}$, {\color{red}Proposition \ref{RepFor} immediately allows us to conclude: for any $(c_1,\ldots,c_n) \in \mathbb S_{(a_1,\ldots,a_n)}$
		%
		%}
	\noindent If $P$ vanishes at $(b_1,\ldots,b_n) \in \mathbb C_K^n\cap \mathbb S_{(a_1,\ldots,a_n)}$ and $(c_1,\ldots,c_n)\in \mathbb C_J^n\cap \mathbb S_{(a_1,\ldots,a_n)}$, with $J\neq K$, then using again Proposition \ref{RepFor}, we can write
	\[0=P(c_1,\ldots,c_n)=\frac{1-JK}{2}P(b_1,\ldots,b_n)+\frac{1+JK}{2}P(\bar b_1,\ldots,\bar b_n)=\frac{1+JK}{2}P(\bar b_1,\ldots,\bar b_n).\]   
	Since $J\neq K$, we get that $P(\bar b_1,\ldots,\bar b_n)=0$. From the previous step we conclude that $P$ vanishes on the entire $\mathbb S_{(a_1,\ldots,a_n)}$.
\end{proof}

{\noindent As a consequence of Proposition \ref{zericinesi}, we get the following result which is a generalization of Lemma \ref{sfera1} and the first step in the enlarging procedure of the zero set. 
\begin{lemma}\label{sfere}
	If $I$ is a right ideal in $E_{(a_1,\ldots,a_t,a_{t+1},\ldots,a_n)}$ with $a_la_m=a_ma_l$ for any $l,m \le t$ and $a_ta_{t+1}\neq a_{t+1}a_t$, then $I\subset E_{\mathbb S_{(a_1,\ldots, a_t)}\times \{a_{t+1}\}\times \cdots \times \{a_n\}}$.
\end{lemma}
\begin{proof}
	Let $P\in I$. Then $P*q_{t+1}= q_{t+1}*P \in I$. By direct computation, arguing as in the proof of Lemma \ref{sfera1}, we get that 
	\[q_{t+1}*P(q_1,\ldots,q_n)=q_{t+1}\cdot
		P(q_{t+1}^{-1}q_1q_{t+1},\ldots, q_{t+1}^{-1}q_t
		q_{t+1},q_{t+1}, \ldots, q_n)\]
%	 if 
%	\[P(q_1,\ldots,q_n)=\sum_{(k_1,\ldots,
%					k_n)\in K} q_1^{k_1}q_2^{k_2}\cdots q_n^{k_n} a_{k_1,\ldots, k_n},\]
%then	
%%	=\sum_{(k_1,\ldots, k_n)\in K} q_1^{k_1}q_2^{k_2}\cdots
%%			Q_{a_{\ell}}(q_l)q_l^{k_l} \ldots q_n^{k_n} a_{k_1,\ldots, k_n}\] 
%		%{\color{red} Thanks to Proposition \ref{starmon}}
%	\begin{equation}
%		\begin{aligned}
%	q_{t+1}*P(q_1,\ldots,q_n)&=\sum_{(k_1,\ldots,
%		k_n)\in K} q_1^{k_1}\cdots q_t^{k_t} q_{t+1}^{k_{t+1}+1}\cdots q_n^{k_n} a_{k_1,\ldots, k_n}\\
%		&=q_{t+1}\cdot
%	P(q_{t+1}^{-1}q_1q_{t+1},\ldots, q_{t+1}^{-1}q_t
%	q_{t+1},q_{t+1}, \ldots, q_n)
%		\end{aligned}
%	\end{equation}
	 when $q_{t+1}\neq 0$.  Since
	$q_{t+1}*P$ vanishes at $(a_1,\ldots,a_n)$, we get that $P$
	vanishes both at $(a_1,\ldots,a_t,a_{t+1},\ldots, a_n)$ and at
	$(\tilde a_1,\ldots, \tilde a_t, a_{t+1},\ldots,a_n)$ where $\tilde
	a_\ell ={a_{t+1}}^{-1} a_\ell a_{t+1}$.  Applying Proposition
	\ref{zericinesi} to $\hat
	P(q_1,\ldots,q_t)=P(q_1,\ldots,q_t,a_{t+1},\ldots,a_n)$ we get
	that $\hat P$ is zero on the entire $\mathbb S_{(a_1,\ldots,
		a_t)}$.  Thus $P \in E_{\mathbb S_{(a_1,\ldots, a_t)}\times
		\{a_{t+1}\}\times \cdots \times\{ a_n\}}$.
\end{proof}

}
{ %We point out that the previous lemma is the crucial ingredient in the proof of our main result 
	\noindent To proceed in the construction of a larger zero set, we need the following step.}
\begin{lemma}\label{pippo}
		Let $I$ be a right ideal in $E_{\mathbb S_{(a_1,\cdots,a_k)}\times \{a_{k+1}\} \times \cdots\times \{ a_{n}\}}$, 
where $\mathbb S_{(a_1,\cdots,a_k)}$ is an arranged spherical set,
and let $t\in\{k+1,\ldots,n\}$ be such that $a_{l} a_m=a_ma_{l}$ for all $ k+1\leq  l,m\leq t$ and $a_{t}a_{t+1}\neq a_{t+1}a_t$. 
Then, there exists $(\widetilde{a_1},\ldots, \widetilde{a_k})\in \mathbb S_{(a_1,\cdots,a_k)}$  such that 
$\widetilde{a_1},\ldots, \widetilde{a_k}, a_{k+1}, \ldots, a_t$ commute and 
		\[I\subset E_{\mathbb S_{( \widetilde{a_1},\ldots, \widetilde{a_k},{a_{k+1}},\ldots, {a_t})}\times \{ a_{t+1} \}\times \cdots \times\{ a_{n}\} }\cap  E_{\mathbb S_{( \overline{\widetilde{a_1}},\ldots, \overline{\widetilde{a_k}},{a_{k+1}},\ldots, {a_t})}\times \{ a_{t+1} \}\times \cdots \times\{ a_{n}\} }. \]  
% with $\widetilde{a_l}\in\mathbb{S}_{a_l}$ for any $ l=k+1,\ldots, t$ and 
%where
%$\mathbb S_{( \widetilde{a_1},\ldots, \widetilde{a_k},{a_{k+1}},\ldots, {a_t})}$
%$\mathbb S_{(a_1,\cdots,a_k, \widetilde{a_{k+1}}, \cdots,\widetilde{a_t})}$ 
%is an arranged spherical set.

\end{lemma} 
\begin{proof}
	Let $I \subset E_{\mathbb S_{(a_1,\cdots,a_k)}\times \{a_{k+1}\} \times \cdots\times \{ a_{n}\}}$. {Then there exist exactly two points  
	$(\widetilde{a_1},\ldots, \widetilde{a_k})$ and $(\overline{\widetilde{a_1}},\ldots, \overline{\widetilde{a_k}})\in \mathbb S_{(a_1,\cdots,a_k)}$ 
	%	$\widetilde {a_{k+1}},\ldots,\widetilde {a_t} \in \mathbb S_{(a_{k+1},\ldots,a_t)}$ 
whose components commute with 
	%$a_{1},\ldots,a_k$, 
		$a_{k+1}, \ldots, a_t$ so that
	 $$I\subset E_{( \widetilde{a_1},\ldots, \widetilde{a_k}, {a_{k+1}},\ldots, {a_t}, a_{t+1},\ldots, a_n)}\cap E_{( \overline{\widetilde{a_1}},\ldots, \overline{\widetilde{a_k}}, {a_{k+1}},\ldots, {a_t}, a_{t+1},\ldots, a_n)}.$$}
	Since $a_{t+1}a_t\neq a_ta_{t+1}$, from Lemma \ref{sfere} we can then conclude that 
			\[I\subset E_{\mathbb S_{( \widetilde{a_1},\ldots, \widetilde{a_k}, {a_{k+1}},\ldots, {a_t})}\times \{ a_{t+1} \}\times \cdots \times\{ a_{n}\} } \cap  E_{\mathbb S_{( \overline{\widetilde{a_1}},\ldots, \overline{\widetilde{a_k}},{a_{k+1}},\ldots, {a_t})}\times \{ a_{t+1} \}\times \cdots \times\{ a_{n}\} }.\]  

\end{proof}
	\noindent Iterating the procedure described in the proof of Lemma \ref{pippo}, we get the following result, which, with a suitable change of notations, has been proved by Alon and Paran and used as the main tool to prove Theorem 1.1 in \cite{israeliani-1}. 
\begin{proposition}\label{caso2}
Assume  $I$ is a right ideal in $E_{\mathbb S_{(a_1,\cdots,a_k)}\times \{a_{k+1}\} \times \cdots\times \{  a_{n}\}}$, 
with $\mathbb S_{(a_1,\cdots,a_k)}$  an arranged spherical set.
Let $1<t_1=k<t_2\ldots<t_{p}<n$ be such that $$a_1,\ldots,a_{t_1}\in \mathbb{C}_{J_1}, \quad 
a_{t_1+1},\ldots,a_{t_2}\in \mathbb{C}_{J_2},\quad ..., \quad a_{t_p+1},\ldots,a_{n}\in \mathbb{C}_{J_{p+1}},$$
where $J_\ell \in \mathbb S$, for all $\ell=1,\ldots,p+1$ and $J_\ell\cdot J_{\ell+1}\neq J_{\ell+1}\cdot J_{\ell}$ for all $\ell=1,\ldots,p.$
%
%If $t=\min\{ s : s>k \ \text{ and } \ a_la_m=a_ma_l \text{ for any } \ s\le l,m \le n \}$ and 
%$({a_1},\ldots, {a_k}, {a_{k+1}}, \cdots, {a_t})$ consists of $p$ blocks,
%then 
Then 
there exist  $2^{p-1}$ distinct arranged spheres   
 $\mathbb S_{ (\widetilde{a_1}, %\ldots, \widetilde{a_k}, \widetilde{a_{k+1}}, 
 \ldots, \widetilde{a_{t_p}})}$, 
 %where $p$ is the number of  adjacent blocks of subset non-commuting components in the string $(a_1,\cdots, a_t)$, 
 such that 
\[I\subset E_{\mathbb S_{(\widetilde{a_1},%\ldots, \widetilde{a_k}, \widetilde{a_{k+1}}, 
\ldots, \widetilde{a_{t_p}})}\times \{a_{{t_p}+1} \}\times \cdots\times\{  a_{n}\}},\]
 with $\widetilde{a_l}\in\mathbb{S}_{a_l}$ for any $ l=1,\ldots, t_p$. 
% and 
% $\mathbb S_{(\widetilde{a_1},\ldots 
% %\widetilde{a_{k+1}}, \cdots,
% \widetilde{a_t})}$ is an arranged spherical set.
% \vskip 1cm 
 
% 
%\noindent {\color{red}
%  there exist  $2^{p-2}$ arranged spheres   of the form
% $\mathbb S_{ ({a_1}, \ldots, {a_k}, \widetilde{a_{k+1}}, 
% \ldots, \widetilde{a_t})}$, 
% and $2^{p-2}$ arranged spheres   of the form
% $\mathbb S_{ (\overline {a_1}, \ldots, \overline{a_k}, \widetilde{a_{k+1}}, 
% \ldots, \widetilde{a_t})}$, 
% with $\widetilde{a_l}\in\mathbb{S}_{a_l}$ for any $ l=k+1,\ldots, t$ 
% %where $p$ is the number of  adjacent blocks of subset nn-commuting components in the string $(a_1,\cdots, a_t)$, 
% such that 
%\[I\subset E_{\mathbb S_{({a_1},\ldots, {a_k}, \widetilde{a_{k+1}}, 
%\ldots, \widetilde{a_t})}\times \{a_{t+1} \}\times \cdots\times\{  a_{n}\}},\]
% and
% \[I\subset E_{\mathbb S_{(\overline{a_1},\ldots, \overline{a_k}, \widetilde{a_{k+1}}, 
%\ldots, \widetilde{a_t})}\times \{a_{t+1} \}\times \cdots\times\{  a_{n}\}}.\]
%
% }
 
\end{proposition}
\begin{proof}
We can proceed as in the proof of Lemma \ref{pippo} for $p-1$ steps: at each step we can choose two different arranged spheres. 
%$\mathbb S_{(\)}on which the polynomials  
\end{proof}

	\noindent The results on common zeros of slice regular polynomials in a right ideal of $\mathbb H[q_1,\ldots,q_n]$ obtained in the previous section, lead to the following definition.
\begin{definition}\label{BS}
Let $U$ be a subset of $\mathbb H^n$. We will say that $U$ is \emph{balloon symmetric} if whenever $(a_1,\ldots,a_n)\in U$, with 
%with $a_\ell a_m\neq a_m a_\ell$ for some $\ell,m\in\{1,\ldots,n\}$, if  
$t=\min\{ k : \ a_ra_s=a_sa_r \text{ for any } \ r,s\in\{k+1,\ldots,n \}\}$ greater than 1, then, for at least one  arranged spherical set 
$\mathbb S_{(\widetilde{a_1},\cdots, \widetilde{a_t })}$ with $\widetilde{a_l}\in\mathbb{S}_{a_l}, \  l=1,\ldots, t$, the 
\emph{balloon} $\mathbb{S}_{(\widetilde{ a_1},\ldots,\widetilde{ a_t})}\times\{a_{t+1}\}\times\cdots\times \{a_n\}$ is contained in $U$.

% with $\widetilde{a_l}\in\mathbb{S}_{a_l}$ for any $ l=1,\ldots, t$ and 
%$\mathbb S_{(\widetilde{a_{1}}, \cdots,\widetilde{a_t})}$ is an arranged spherical set.
%%where  $\tilde a_\ell$ belongs to $\mathbb S_{a_\ell}$ and commutes with $a_{n}$ for any $\ell$, .

\end{definition}
{\noindent Notice that if all points in $U\subset \mathbb H^n$ have commuting components, then $U$ is trivially balloon symmetric.}

\noindent In the two variable case, the balloon symmetry coincides with the notion of {\em $q_1$-symmetry} introduced in \cite[Definition 4.11]{Nul1}.

\begin{remark}\label{balloon}
Proposition \ref{caso2} yields that if $I$ is a right ideal in $\mathbb H[q_1,\ldots,q_n]$, then $\mathcal V(I)$ is balloon symmetric.
\end{remark}

\section{Strong Version of Hilbert Nullstellensatz  and Slice Algebraic Sets in several quaternionic variables}

\noindent The geometric properties of  vanishing sets of right ideals of slice regular polynomials are crucial to prove the following result, which is the first step to generalize the strong version of the Hilbert Nullstellensatz  
 in $n>2$  variables.

\begin{theorem}\label{Siannullano}
Let $I$ be a right ideal in $\mathbb H[q_1,\ldots,q_n]$, then the right ideal $\mathcal J(\mathcal V_c(I))$, generated by slice regular polynomials vanishing on $\mathcal V_c(I)$, is contained in the set of slice regular polynomials vanishing on $\mathcal V(I)$. 
	\end{theorem}
\begin{proof}
%The statement of the Theorem is equivalent to showing that
 %$\mathcal J(\mathcal V_c(I))=\mathcal J(\mathcal V(I))$; 
%Since  $\mathcal J(\mathcal V(I))\subseteq\mathcal J(\mathcal V_c(I))$ for any right ideal $I$  in $\mathbb H[q_1,\ldots,q_n]$,
%we are left to prove the other inclusion.
Let us prove the inclusion by induction on the number of variables.
For $n=1$, recalling equality \eqref{jvci} there is nothing to prove, since $\mathcal V (I)=\mathcal V_c(I)$ for any right ideal $I$ in $\mathbb H[q]$.

\noindent 
Suppose now that the inclusion holds for any right ideal in $\mathbb H[q_1,\ldots,q_{r}]$ with $r<n$.
%namely $\mathcal J(\mathcal V_c(\hat I))\subseteq\mathcal J(\mathcal V(\hat I))$ 
% for any (right) ideal $\hat I$ in $\mathbb H[q_1,\ldots,q_{r}]$
\noindent Let $I$ be a right ideal in $\mathbb H[q_1,\ldots,q_{n}]$, let $(a_1,\ldots,a_n) \in \mathcal V(I)\setminus	\mathcal V_c(I)$ and
consider $P_0\in \mathcal J(\mathcal V_c(I))$. 
Let $t=\min\{ k : a_\ell a_{m}=a_{m}a_\ell \text{ for any } \ell,m > k\}$. Notice that $t\ge 1$ { since $(a_1,\ldots,a_n)\notin \mathcal V_c(I)$ and $t\le n-1$ since $a_n$ commutes with itself.}

%Thanks to Remark \ref{balloon}, $\mathcal V(I)$ is balloon symmetric, and hence it contains a balloon of the form $\mathbb S_{(\widetilde{ a}_1,\ldots,\widetilde{ a}_t)}\times \{a_{t+1}\}\times \cdots \times \{a_n\} $ where $\widetilde {a_\ell}$ belongs to $\mathbb S_{a_\ell}$ and commutes with $a_{n}$ for any $\ell$. 
%Therefore $\mathcal V_c(I)$ contains two points of such balloon: $(\tilde a_1,\ldots,\tilde a_t,a_{t+1}, \ldots ,a_n)$ and $(\overline{\widetilde {a_1}},\ldots,\overline{\widetilde {a_t}},a_{t+1}, \ldots ,a_n)$; as a direct application of  Proposition \ref{zericinesi}, we have that ${P_0}$ vanishes on the entire balloon $\mathbb S_{(\widetilde{a_1},\ldots,\widetilde{ a_t})}\times \{a_{t+1}\}\times \cdots \times \{a_n\} $.

\noindent Let us define $\hat I$ as
\[\hat I = \{\hat P \in \mathbb{H}[q_1,\ldots,q_t] \ : \ \hat P(q_1,\ldots,q_t)=P(q_1,\ldots,q_t,a_{t+1},\ldots,a_n) \ \text{ with } \ P\in I \}.\]
The set $\hat I$ is a right ideal in $\mathbb{H}[q_1,\ldots,q_t]$. In fact if $\hat P_1, \hat P_2 \in \hat I$, then $\hat P_1 + \hat P_2= \widehat{P_1+P_2}$ and hence it belongs to $\hat I$. Consider now $Q \in \mathbb{H}[q_1,\ldots,q_t]\subset \mathbb H[q_1,\ldots,q_n]$ and $\hat P \in \hat I$. Then
$\hat P * Q = \widehat{P*Q}$ since $Q$ does not depend on $q_{t+1}, \ldots,q_n$, therefore it belongs to $\hat I$. \\
\noindent We want to prove that $\hat P_0 \in \mathcal J(\mathcal V_c(\hat I)).$\\ Let $(b_1,\ldots,b_t)\in  \mathcal V_c(\hat I).$
Observe that $(b_1,\ldots,b_t,a_{t+1},\ldots,a_n)\in \mathcal V(I),$ indeed, for any $P\in I,$ we have $$P(b_1,\ldots,b_t,a_{t+1},\ldots,a_n)=\hat{P}(b_1,\ldots,b_t)=0,$$ since $\hat{P}\in\hat{I}$. Now there are two possibilities:
\noindent
 if $(b_1,\ldots,b_t,a_{t+1},\ldots,a_n)\in \mathcal V_c(I)$ we have that $$\hat{P_0}(b_1,\ldots,b_t)=P_0(b_1,\ldots,b_t,a_{t+1},\ldots,a_n)=0.$$ 
 %That is $\hat P_0 \in \mathcal J(\mathcal V_c(\hat I)).$\\
\noindent   If,  otherwise, $(b_1,\ldots,b_t,a_{t+1},\ldots,a_n)\in \mathcal V(I)\setminus \mathcal V_c(I)$, from Lemma \ref{sfere}, we have that the balloon $$\mathbb S_{( b_1,\ldots,b_t)}\times \{a_{t+1}\}\times \cdots \times \{a_n\} \subseteq \mathcal{V}(I).$$
Therefore there exist two points $$(\widetilde{b_1},\ldots,\widetilde{b_t},a_{t+1},\ldots,a_n)\quad\mbox{and}\quad(\overline{\widetilde{ b_1}},\ldots,\overline{\widetilde{ b_t}},a_{t+1},\ldots,a_n)$$ in $\mathbb S_{( b_1,\ldots,b_t)}\times \{a_{t+1}\}\times \cdots \times \{a_n\}\cap \mathcal{V}_c(I)$ on which $P_0$ vanishes. 
%That is $$\hat{P}_0({\bf{ b_1}})=\hat{P}_0({\bf{ b_2}})=0.$$ 
Thus, thanks to Proposition \ref{zericinesi} for $\hat P_0,$ we get that $\hat{P}_0$ vanishes on the entire $\mathbb S_{({\widetilde{ b_1}},\ldots,{\widetilde{ b_t}})}$, and thus on $( b_1,\ldots,b_t)$. Therefore  we have that  $\hat P_0$ vanishes on $\mathcal V_c(\hat I)$ i.e.  $\hat P_0\in\mathcal{J}(\mathcal{V}_c(\hat I)).$ By induction, we have that $\hat P_0\in\mathcal{J}(\mathcal{V}(\hat I)).$\\ Now observe that $( a_1,\ldots,a_t)\in \mathcal V(\hat I).$ Indeed  for any $\hat Q\in\hat I$ we have $$\hat Q(a_1,\ldots,a_t)=Q(a_1,\ldots,a_n)$$ for $Q\in I$. Since $(a_1,\ldots,a_n)\in\mathcal{V}(I)$ we have  $\hat Q(a_1,\ldots,a_t)=0.$ Thus, since  $\hat{P}_0$ vanishes on $\mathcal V(\hat I)$ we get
$$P_0(a_1, \ldots, a_n)=\hat{P}_0(a_1,\ldots,a_t)=0.$$
% that $P_0$ vanishes on $(a_1, \ldots, a_n)$. % $\mathcal V(I)$.
Since $(a_1,\ldots,a_n)$ was chosen arbitrarily in $\mathcal V(I)\setminus \mathcal V_c(I)$,
we conclude that $P_0$ vanishes on the entire $\mathcal V(I)$, i.e.  $P_0\in\mathcal{J}(\mathcal V(I))$.
\end{proof}
%since $P_0 \in \mathcal J(\mathcal V_c( I))$. By induction, we have therefore that $\hat P_0$ vanishes on $\mathcal V(\hat I)$. Hence we get
 %\[P_0(a_1,\ldots,a_t,a_{t+1}, \ldots, a_n)=\hat P_0 (a_1,\ldots,a_t)=0
%\]
%as $(a_1,\ldots,a_t)\in \mathcal V(\hat I)$.
%Using this caracterization of  $\mathcal{J}(\mathcal V(I))$,
\noindent The previous result was independently proven in \cite[Theorem 1.1]{israeliani-3}. 
% Theorem \ref{Siannullano} yields the other inclusion,
\noindent Since  $\mathcal J(\mathcal V(I))\subseteq\mathcal J(\mathcal V_c(I))$ for any right ideal $I$ 
and Theorem \ref{Siannullano} yields the other inclusion,
we can extend Theorem 4.15 in \cite{Nul1}
to right ideals of slice regular polynomials in $n>2$ quaternionic variables.

%and Theorem \ref{Siannullano} yields the other inclusion,
%we can extend Theorem 4.15 in \cite{Nul1}
%to ideals of slice regular polynomials in several quaternionic variables.

\begin{corollary}\label{idealiuguali}
	Let $I\in \mathbb{H}[q_1,\ldots,q_n]$ be a right ideal. Then $\mathcal J(\mathcal V(I))=\mathcal J(\mathcal V_c(I))$.
\end{corollary}

\noindent As a consequence of Corollary \ref{idealiuguali} we get the following strong version of the Hilbert Nullstellensatz  for right ideals of slice  regular polynomials in $n$ quaternionic variables.
\begin{theorem}\label{StrongNull}
Let $I$ be a right ideal in $\mathbb{H}[q_1,\ldots,q_n]$. Then
\[\mathcal{J}(\mathcal{V}(I))=\sqrt{I}.\] 
\end{theorem}
\begin{proof}
Theorem 4.10 in \cite{Nul1}, combined with equation $(4.2)$ in \cite{Nul1}, imply that 
$\mathcal{J}(\mathcal{V}_c(I))=\sqrt{I}.$
Corollary \ref{idealiuguali} immediately leads to the conclusion.
\end{proof}
%% \marginpar{if needed!}
%% {
%% The same holds whenever we start from a balloon symmetric set : $V$ balloon symmetric $
%% \implies$ $\mathcal J(V)=\{P \ |\ P(q)=0\;\; \forall q \in V\}$.}
%% \marginpar{check if needed!}
%\subsection{Algebraic sets}

\noindent We can now introduce the following important
\begin{definition}
	A subset $V\subseteq \mathbb H^n$ is called {\em slice algebraic} if for any $K\in \mathbb S$, $V\cap \mathbb C_K^n$ is a complex algebraic subset of $\mathbb C_K^n$.
\end{definition}

\begin{proposition}

If one defines slice algebraic subsets as {\em closed} subsets of  $\mathbb H^n$,
then one obtains a topology on $\mathbb H^n$ (which resembles the Zariski topology).

\end{proposition}
\begin{proof}
Let us show that the family of slice algebraic sets satisfies the topology axioms.  
\begin{enumerate}
	\item The empty set is closed; indeed,  for any $K\in \mathbb S$,
          $\varnothing\cap \mathbb C_K^n=\varnothing$, which is
          algebraic in $\mathbb C_K^n$.
	\item $\mathbb H^n$ is closed; indeed, for any $K\in \mathbb S$,
          $\mathbb H^n\cap \mathbb C_K^n=\mathbb C_K^n$, which is
          algebraic in $\mathbb C_K^n$.
	\item If $V,W$ are slice algebraic sets, then for any $K\in \mathbb S$, 
	\[(V\cup W )\cap \mathbb C_K^n=(V\cap \mathbb C_K^n)\cup (W \cap \mathbb C_K^n),\]
	is the union of two algebraic sets in $\mathbb C_K^n$, thus $V\cup W$  is a slice algebraic set
        in $ \mathbb H^n$.
        
	\item If $\{V_\ell\}_\ell$ is a  family of slice algebraic
          sets, then, for any $K\in \mathbb
          S$, it turns out that
          \[(\bigcap_{\ell} V_\ell)\cap \mathbb
          C_K^n=\bigcap_{\ell}(V_\ell\cap \mathbb C_K^n),\] is the
          intersection of an infinite family of algebraic sets in
          $\mathbb C_K^n$, thus $\bigcap\limits_{\ell} V_\ell$
          is a slice algebraic set in $ \mathbb H^n$.
\end{enumerate}
\end{proof}
%\noindent{ In order to prove that $\mathcal{V}(I)$ turns out to be an algebraic set  in $\mathbb{H}^n$ we need to apply Lemma \ref{splitting}.}
\noindent An important link between slice algebraic sets and vanishing sets of slice regular polynomials which belong to a right ideal   in $\mathbb H [q_1,\ldots,q_n]$ is the following.
\begin{theorem}
	Let $I$ be a right ideal in $\mathbb H [q_1,\ldots,q_n]$, then
        $\mathcal V(I)$ is a slice algebraic set in $ \mathbb H^n$.
\end{theorem}
\begin{proof}
Since $\mathbb H [q_1,\ldots,q_n]$ is Noetherian, there exist $P_1,\ldots,P_m\in \mathbb{H}[q_1,\ldots, q_n]$ such that $I=\langle P_1,\ldots,P_m\rangle$.
Let $K\in \mathbb S$ 
%and take $\underline{q}
% \in \mathcal V(I)\cap \mathbb (C_K)^n$. Then $P_1(\underline{q})=\ldots=P_m(\underline{q})=0$. 
  and let $L\in \mathbb S$ be orthogonal to $K$ and consider the splitting \eqref{split} of the generators of $I$ on $\mathbb C_K^n$ with respect to $L$:
 $P_\ell({\bf{q}})=F_\ell({\bf{q}})+G_\ell({\bf{q}})L$, 
  with $F_\ell$ and $G_\ell$ complex polynomials in $\mathbb C_K^n$, for all $\ell=1,\ldots,m$. Hence  ${\bf{q}}
 \in \mathcal V(I)\cap (\mathbb{C}_K)^n$ if and only if $P_1({\bf{q}})=\ldots=P_m({\bf{q}})=0$, if and only if 
  $F_\ell({\bf{q}})=G_\ell({\bf{q}})=0$ for all $\ell=1,\ldots,m$. 
\end{proof}
\begin{remark}\label{ij}
	Not every slice algebraic set can be written as $\mathcal V (I)$ for some right ideal $I$ in $\mathbb H[q_1,\ldots,q_n]$.
	As an example consider $V=\{i,j\}\subset \mathbb H$. It is easy to see that it is slice algebraic, but if a polynomial in $I\subset \mathbb H[q]$ vanishes on $i$ and $j$, then it vanishes on the entire sphere $\mathbb S$. Thus $\mathcal V(I)$ should contain $\mathbb S$. 
\end{remark}
\noindent The operators $\mathcal V$ and $\mathcal J$ establish a relation between right ideals and slice algebraic sets. Such a relation becomes a bijection if we restrict to radical right ideals and to algebraic sets of the form $\mathcal V(I)$.

\begin{proposition}
The operators $\mathcal V$ and $\mathcal J$ are such that:
	
\begin{enumerate}
	
	\item for any radical right ideal $I$ in $\mathbb H[q_1,\ldots,q_n]$,  $\mathcal J (\mathcal V (I))=I$;
	\item for any %slice algebraic set $V$ in $\mathbb H^n$ of the form $\mathcal V(I)$, $\mathcal V(\mathcal J(V))=V$. 
right ideal $I$ in $\mathbb H[q_1,\ldots,q_n]$,  $\mathcal{V}(\mathcal J (\mathcal V (I)))=\mathcal{V}(I)$.
\end{enumerate}
\end{proposition}
\begin{proof}
The first statement is a direct consequence of the Strong Nullstellensatz \ref{StrongNull}.

\noindent For the second statement, thanks to Proposition \ref{Siannullano}, $\mathcal J (\mathcal{V}(I))$ coincides with the set of slice regular polynomials which vanish on $\mathcal{V}(I)$. 
Hence $\mathcal{V}(I) \subseteq \mathcal V(\mathcal J (\mathcal{V}(I)))$.
\noindent On the other hand,  any $P \in I$  vanishes on $\mathcal{V}(I)$. Hence $I \subseteq \mathcal J (\mathcal{V}(I))$ and therefore $\mathcal V(I)\supseteq \mathcal V(\mathcal J (\mathcal{V}(I)))$.
\end{proof}

\begin{remark}
If $V$ is not of the form $\mathcal V(I)$, then $\mathcal V (\mathcal J(V))$ does not coincide with $V$. Moreover it is not possible to establish an inclusion relation between the two sets:  
\begin{enumerate}
	\item the example mentioned in Remark \ref{ij} shows a case in which $\mathcal V(\mathcal J (V))\supsetneq V$;
	\item starting from the slice algebraic set $V=\{(i,j)\}\in \mathbb H^2$, we get  (see \cite{israeliani-1, Nul1}) that $\mathcal J (V)=\mathbb H[q_1,q_2]$ and hence $\mathcal V (\mathcal J(V))=\varnothing \subsetneq V$.
\end{enumerate}

\end{remark}

\section*{Declarations}
\noindent The authors have no competing interests to declare that are relevant to the content of this article.\\

%
%
%\begin{enumerate}
%	\item Cercare di capire come sono fatti questi polinomi nel pallone.
%	\item Come definire le variet\`a algebriche regolari. Come luoghi di zeri di ideali?
%\end{enumerate}

\end{document}